\documentclass[11pt]{amsart}
\usepackage{amsmath,amssymb,amsthm}
\usepackage[latin1]{inputenc}
\usepackage{version,tabularx,multicol}
\usepackage{graphicx,float,psfrag}

\newcommand{\reff}[1]{(\ref{#1})}

\theoremstyle{plain}
\newtheorem{theo}{Theorem}[section]
\newtheorem{cor}[theo]{Corollary}
\newtheorem{prop}[theo]{Proposition}
\newtheorem{lem}[theo]{Lemma}
\newtheorem{defi}[theo]{Definition}
\theoremstyle{remark}
\newtheorem{rem}[theo]{Remark}

\newcommand{\cc}{{\mathcal C}}

\newcommand{\cf}{{\mathcal F}}
\newcommand{\cg}{{\mathcal G}}

\newcommand{\ci}{{\mathcal I}}

\newcommand{\cl}{{\mathcal L}}
\newcommand{\cn}{{\mathcal N}}

\newcommand{\ct}{{\mathcal T}}
\newcommand{\cu}{{\mathcal U}}

\newcommand{\cy}{{\mathcal Y}}

\newcommand{\E}{{\mathbb E}}

\newcommand{\M}{{\mathbb M}}
\newcommand{\N}{{\mathbb N}}

\newcommand{\R}{{\mathbb R}}

\newcommand{\T}{{\mathbb T}}

\newcommand{\ind}{{\bf 1}}

\newcommand{\inv}[1]{\mathop{\frac{1}{ #1}}\nolimits}
\newcommand{\expp}[1]{\mathop {\mathrm{e}^{ #1}}}

\newcommand{\lb}{[\![}
\newcommand{\rb}{]\!]}

\begin{document}

\title{Record process on the Continuum Random Tree}

\date{\today}
\author{Romain Abraham} 

\address{
Romain Abraham,
Laboratoire MAPMO, CNRS, UMR 7349,
F\'ed\'eration Denis Poisson, FR 2964,
 Université d'Orléans,
B.P. 6759,
45067 Orléans cedex 2,
France.
}
  
\email{romain.abraham@univ-orleans.fr}

\author{Jean-François Delmas}

\address{
Jean-Fran\c cois Delmas,
Université Paris-Est, \'Ecole des Ponts, CERMICS, 6-8
av. Blaise Pascal, 
  Champs-sur-Marne, 77455 Marne La Vallée, France.}

\email{delmas@cermics.enpc.fr}

\thanks{This work is partially supported by the ``Agence Nationale de
  la Recherche'', ANR-08-BLAN-0190.}

\begin{abstract}
By considering a continuous pruning procedure on Aldous's Brownian
tree, we construct a random variable $\Theta$ which is distributed,
conditionally given the tree, according to the probability law
introduced by Janson as the limit distribution of the number of cuts
needed to isolate the root in a critical Galton-Watson tree. We also
prove that this random variable can be obtained as the a.s. limit of
the number of cuts needed to cut down the subtree of the continuum
tree spanned by $n$ leaves. 
\end{abstract}

\keywords{continuum random tree, records, cutting down a tree}

\subjclass[2010]{60J80,60C05}

\maketitle

\section{Introduction}
\label{sec:intro}

The problem of randomly cutting a rooted tree arises first in Meir and
Moon \cite{mm:cdrt}. Given a rooted tree $T_n$ with $n$ edges, select
an edge uniformly at random and delete the subtree not containing the
root attached to this edge. On the remaining tree, iterate this
procedure until only the edge attached to the root is left. We denote
by $X_n$ the number of edge-removals needed to isolate the root. The
problem is then to study asymptotics of this random number $X_n$,
depending on the law of the initial tree $T_n$.

In the original paper \cite{mm:cdrt}, Meir and Moon considered Cayley
trees and obtained asymptotics for the first two moments of
$X_n$. Limits in distribution were then obtained by Panholzer
\cite{p:cdvst} for some simply generated trees, by Drmota,
Iksanov, Möhle and Roesler \cite{dimr:ldncnirrrt} for random
recursive trees, by Holmgren \cite{h:rrcbsc} for binary search trees,
by Bertoin \cite{b:ft} for Cayley trees and by Janson \cite{j:rcrdrt}
for conditioned Galton-Watson trees. The main result of
\cite{j:rcrdrt} states  that,  if  the offspring  distribution  of  the
Galton-Watson process  is critical (that is  with mean equal  to 1) with
finite  variance, which  we take  equal to  1 for  simplicity,  then the
following convergence in distribution  of the conditional laws (specified
by their moments) holds:
\begin{equation}
   \label{eq:TzT}
\cl(
X_n/\sqrt{n}\,|\, T_n/\sqrt{n})
\; \xrightarrow[n\rightarrow +\infty ]{(d)} \;
\cl(Z_\ct\,|\,\ct)
\end{equation}
where  $\ct$  is the so-called continuum  random  tree  (CRT) introduced  by
Aldous  \cite{a:crt1,a:crt3} and can  be seen  as the  limit in
distribution of  $T_n/\sqrt{n}$ (see \cite{a:crt3}).  Furthermore, the
random  variable  $Z_\ct$  has (unconditional)   Rayleigh  distribution  with  density
$x\expp{-x^2/2}\ind_{\{x>0\}}$.   However,   there  is  no  constructive
description of $Z_\ct$ conditionally on $\ct$.

The first goal of the paper is to give a continuous pruning procedure
of the CRT
that leads to a random variable that is indeed distributed,
conditionally given the tree, as  $Z_\ct$. In order to better
understand the intuitive idea of the record process on the CRT, let us
first consider the pruning of the simple tree consisting in the segment
$[0,1]$ divided into $n$ segments of equal length, rooted at 0. Select
an edge at random and discard what is located on the right of this
edge. Then chose again an edge at random on the remaining segments and
iterate the procedure until the segment attached to 0 is chosen. It is
clear that the continuous analogue of this procedure (when the
number $n$ of segments tends to $+\infty$) is the so-called
stick-breaking scheme: consider a uniform random variable $U_1$ on
$[0,1]$, then conditionally given $U_1$, consider a uniform random
variable $U_2$ on $[0,U_1]$ and so on. The sequence $(U_n)_{n\ge 0}$
corresponds to the successive cuts of the interval $[0,1]$ in the
continuous pruning. Moreover, this sequence can be obtained as the
records of a Poisson point process. More precisely, if we consider a
Poisson point measure $\sum_{i\in I}\delta_{(x_i,t_i)}$ on
$[0,1]\times[0,+\infty)$ with intensity the Lebesgue measure, then the
  sequence $(U_n)$ is distributed as the sequence of jumps of the
  record process
\[
\theta(x)=\inf\{t_i,x_i\in[0,x]\}.
\]

In our case, the limiting object is Aldous's CRT (instead of the segment
$[0,1]$). More precisely, we consider  a real tree $\ct$ associated with
the branching mechanism $\psi(u)=\alpha u^2$ under the excursion measure
$\N$. This tree is coded  by the height process $\sqrt{2/\alpha} B_{ex}$
where $B_{ex}$  is a positive  Brownian excursion. This tree  is endowed
with two  measures: the length  measure $\ell(dx)$ which  corresponds to
the Lebesgue measure  on the skeleton of the tree,  and the mass measure
$m^\ct(dx)$  which   is  uniform  on   the  leaves  of  the   tree.  Let
$\sigma=m^\ct(\ct)$  be   the  total   mass  of  $\ct$.    Aldous's  CRT
corresponds to  the distribution  of the tree  $\ct$ conditioned  on the
total  mass $\sigma=1$,  with $\alpha=1/2$.  We then  add cut  points on
$\ct$  as  above  thanks  to  a  Poisson  point  measure  on  $\ct\times
[0,+\infty)$ with intensity
\[
\alpha \ell(dx)d\theta
\]
 in the same
spirit as in \cite{ap:sac} (see also \cite{as:psf} for a direct
construction, and \cite{adv:plcrt} for the pruning of a general L\'evy
tree). We denote by $(x_i,q_i)$ the atoms of this point measure, $x_i$
represents the location of the cut point and $q_i$ represents the time
at which it appears. For $x\in \ct$, we denote by
\[
\theta(x)=\inf\{q_i,\ x_i\in\lb \emptyset,x\rb\}
\]
where $\lb \emptyset,x\rb\subset \ct$ denotes the path between $x$ and the
root. When a mark appears, we cut the tree on this mark and discard
the subtree not containing the root. Then $\theta(x)$ represents the
time at which $x$ is separated from the root. Then we define
$$\Theta=\int_\ct\theta(x)m^\ct(dx)\qquad\mbox{and}\qquad
Z=\sqrt{\frac{2\alpha}{\sigma}}\Theta.$$
We prove (see Theorem \ref{thm:main}) that, conditionally on $\ct$,
$Z$ and $Z_\ct$ have indeed the same law. The proof of this result
relies on another representation of $\Theta$ in terms of the mass of
the pruned tree (a similar result also appears in Addario-Berry,
Broutin and Holmgren \cite{abbh:cdtmc}). More precisely, if we set
$$\sigma_q=\int_\ct\ind_{\{\theta(x)\ge q\}}m^\ct(dx)$$
the mass of the remaining tree at time $q$, then we have
$$\Theta=\int_0^{+\infty}\sigma_q\, dq.$$

Using this framework, we can extend in some sense Janson's result by
obtaining an a.s. convergence in a special case. We consider,
conditionally given $\ct$, $n$ leaves uniformly chosen (i.e. sampled
according to the mass measure $m^\ct$) and we denote by $T_n$ the
sub-tree of $\ct$ spanned by these $n$ leaves and the root. The tree
$T_n$ is distributed under $\N[\,\, \cdot\mid \sigma=1]$ as a uniform
ordered binary tree with $n$ leaves (and hence $2n-1$ edges) with
random edge lengths. We denote by $T_n^*$ the tree obtained by
removing from $T_n$ the edge attached to the root, and by $X_n^*$ the
number of discontinuities of the process $(\theta(x),x\in T_n^*)$. The
quantity $X_n^*+1$ represents the number of cuts
needed to reduce the binary tree $T_n$ to a single branch attached to
the root. Notice that in our framework, several cuts may appear on the
same branch, so $X_n^*$ looks like $X_{2n-1}$ for uniform ordered
binary trees but is not exactly the same. Then, we prove in Theorem
\ref{theo:CV-X_n} that $\N$-a.e. or $\N[\, \, \cdot\mid\sigma=1]$-a.s.:
\[
\lim_{n\to+\infty}\frac{X_n^*}{\sqrt{2n}}=Z.
\]
This result can be extended by studying the fluctuations of the
quantity $X_n^*/\sqrt{2n}$ around its limit, this is the purpose of
Hoscheit \cite{h:fnrrbt}. In this setting the fluctuations come from
the approximation of the record process by its intensity, whereas there
is no contribution from the approximation of $\ct$ by $T_n$. 

Using the second representation of $\Theta$ and results from Abraham,
Delmas and Hoscheit \cite{adh:etiltvp} on the pruning of general
L\'evy trees, we also derive a.s. asymptotics  on the sizes
$(\sigma_i,i\in\ci)$ of the removed sub-trees during the cutting
procedure. According to Propositions \ref{prop:cv-sqi} and
\ref{prop:cv-sqi2}, we have $\N$-a.e.
\[
\lim_{n\to+\infty}\frac{1}{\sqrt
  n}\sum_{i\in\ci}\ind_{\{\sigma^i\geq 1/n\}}
=\lim_{n\to+\infty}\sqrt{n}\sum_{i\in\ci}\sigma^i\ind_{\{\sigma^i\le
1/    n\}}=2\sqrt{\frac{\alpha}{\pi}}\Theta.
\]
This result is extended to general L\'evy trees in Abraham and Delmas
\cite{ad:farplt}.

The  paper  is  organized  as  follows.  In  Section  \ref{sec:crt},  we
introduce the  frameworks of  discrete trees and  real trees  and define
rigorously Aldous's CRT, the mark  process and the record process on the
tree. Section  \ref{sec:janson} is devoted to the  identification of the
law   of   $\Theta$   conditionally   given   the   tree.   In   Section
\ref{sec:convergence}, we prove the  a.s. convergence of $X_n^*$ as well
as   the   convergence   results   on   the  masses   of   the   removed
subtrees.  Finally, we  gathered in  Section  \ref{sec:appendix} several
technical lemmas that are needed in  the proofs but are not the heart of
the paper.

\section{The continuum random tree and the mark process}
\label{sec:crt}

\subsection{Real trees}

We recall here the definition and basic properties of real trees. We
refer to Evans's Saint Flour lectures \cite{e:prt} for more details on
the subject.

\begin{defi}
A real tree is a metric space $(\ct,d)$ satisfying the following two
properties for every $x,y\in\ct$:
\begin{itemize}
\item (Unique geodesic) There is a unique isometric map $f_{x,y}$ from
  $[0,d(x,y)]$ into $\ct$ such that $f_{x,y}(0)=x$ and
  $f_{x,y}(d(x,y))=y$.
\item (No loop) If $\varphi$ is a continuous injective map from
  $[0,1]$ into $\ct$ such that $\varphi(0)=x$ and $\varphi(1)=y$, then
$$\varphi([0,1])=f_{x,y}([0,d(x,y)]).$$
\end{itemize}
A rooted real tree is a real tree with a distinguished vertex denoted
$\emptyset$ and called the root.
\end{defi}

We denote by $\lb x,y\rb=f_{x,y}([0,1])$ the range of the mapping
$f_{x,y}$, which  is the unique continuous injective path between $x$ and $y$ in the
tree, and $\lb x,y\lb=\lb x,y\rb\backslash\{y\}$. A point $x\in\ct$ is said to be a leaf if the set
$\ct\setminus\{x\}$ remains connected. We denote by $\mathrm{Lf}(\ct)$ the set
of leaves of $\ct$. The skeleton of the tree is the
set of non-leaves points $\ct\setminus\mathrm{Lf}(\ct)$. As the trace of the Borel $\sigma$-field on
the skeleton of $\ct$ is generated by the intervals $\lb x,y\rb$, 
one can define a length measure denoted by $\ell(dx)$ on a real
tree by:
$$\ell(\lb x,y\rb)=d(x,y).$$

We will  consider here only  compact real trees  and these trees  can be
coded by some continuous function (see \cite{lg:rrt} or \cite{d:ccrtrvf}). We  consider  a  continuous
function  $\zeta\,:\,  [0,+\infty)\rightarrow[0,+\infty)$  with  compact
support  $[0,\sigma]$ and  such  that $\zeta(0)=\zeta(\sigma)=0$.   This
function   $\zeta$  will  be   called  in   the  following   the  height
function. For every $s,t\ge 0$, we set
$$m_{\zeta}(s,t)=\inf_{r\in[s\wedge t,s\vee t]}\zeta(r),$$
and
\[
d(s,t)=\zeta(s)+\zeta(t)-m_\zeta(s,t).
\]
We then define the equivalence relation $s\sim t$ iff $d(s,t)=0$. We
set $\ct$ the quotient space
$$\ct=[0,+\infty)/\sim.$$
The pseudo-distance $d$ induces a  distance on $\ct$ and we keep notation
$d$ for this distance. We denote by $p$ the canonical projection from $[0,+\infty)$ onto
  $\ct$.
The  metric  space
$(\ct,d)$  is a  compact real  tree which can be  viewed as  a rooted real  tree by  setting $\emptyset=
p(0)$.

On such a compact real tree, we define another measure : the mass
measure $m^\ct$ defined as the push-forward of the Lebesgue measure by
the projection $p$. It is a finite measure supported by the leaves of
$\ct$ and its total mass is
$$m^\ct(\ct)=\sigma.$$

This coding is very useful to define random real trees. For instance,
Aldous's CRT is the random real tree coded by $2B_{ex}$ where $B_{ex}$
denotes a normalized Brownian excursion (i.e. a positive Brownian
excursion with duration 1). Here, we will work
under the $\sigma$-finite measure $\N$ which denotes the law of a real
tree coded by an excursion away from 0  of
$\sqrt{\frac{2}{\alpha}}|B|$  where   $|B|$  is  a   standard  reflected
Brownian motion. The tree $\ct$ under $\N$ is then the genealogical tree of a
continuous   state   branching    process   with   branching   mechanism
$\psi(u)=\alpha u^2$ under its canonical measure. In particular, under
$\N$, $\sigma$ has density on $(0,+\infty )$:
\begin{equation}
   \label{eq:densite-s}
\frac{dr}{2\sqrt{\alpha \pi}\, r^{3/2}}\cdot
\end{equation}

We keep parameter
$\alpha$ in order to stay in the framework of \cite{ad:ctvmp}, and give
the result in the setting of Aldous's CRT ($\alpha=1/2$) or of Brownian
excursion ($\alpha=2$).

Using the scaling property of the Brownian motion, there exists a
regular version of the measure $\N$ conditioned on the length of the
height process $\zeta$. We write $\N^{(r)}$ for the probability
measure $\N[\,\cdot\, |\sigma=r]$. In particular, we handle Aldous's
CRT if we work under $\N^{(1)}$ with $\alpha=1/2$.

If  $x_1,\ldots,x_n\in\ct$,  we   denote  by  $\ct(x_1,\ldots,x_n)$  the
subtree  spanned   by  $\emptyset,x_1,\ldots,x_n$,  i.e.   the  smallest
connected  subset  of  $\ct$  that  contains  $x_1,\ldots,x_n$  and  the
root. In other words, we have
$$\ct(x_1,\ldots,x_n)=\bigcup_{i=1}^n\lb\emptyset,x_i\rb.$$
With an abuse of notation, we write for every $t_1,\ldots,t_n\ge 0$,
$\ct(t_1,\ldots,t_n)$ for the subtree  $\ct(p(t_1),\ldots,p(t_n))$.

\subsection{The mark process}

We define now a mark process $M$ on the tree $\ct$. Conditionally given $\ct$,
let $M(dx,dq)$ be a Poisson point measure on $\ct\times [0,+\infty)$
  with intensity $2\alpha\ell(dx)dq$. An atom $(x_i,q_i)$ of
  this random measure represents a mark on the tree $\ct$, $x_i$ is
  the location of this mark whereas $q_i$ denotes the time at which
  the mark appears. 

\begin{rem}
  The coefficient $2\alpha$  in the intensity is added  to have the same
  intensity      as      in      the     pruning      procedures      of
  \cite{as:psf,adv:plcrt,adh:etiltvp} but, as we  shall see, it does not
  appear in the law of the number of records.
\end{rem}

In  fact  we  will  sometimes  work  with  the  restriction  of  $M$  to
$\ct\times[0,a]$  for some  $a>0$. To  simplify the  notations,  we will
always denote by $M$ the mark process (even the restricted one) and will
write  $\M_a^\ct$ for  the law  of $M$  restricted  to $\ct\times[0,a]$,
conditionally    given    $\ct$.     We    also    write    $\N_a[d\ct\,
dM]=\N[d\ct]\M_a^\ct[dM]$,             and            $\N_a^{(r)}[d\ct\,
dM]=\N^{(r)}[d\ct]\M_a^\ct[dM]$.

We set for every $q\ge 0$ and $x\in \ct$:
\begin{equation}
\label{eq:def-pruned}
\theta(x)= \inf\{q>0,\ M(\lb
\emptyset, x\rb\times [0,q])>0\}
\quad\text{and}\quad
\ct_q=\{x\in \ct; \theta(x)\geq q\},
\end{equation}
respectively the first time a mark appears between the root and $x$, and
the tree obtained by pruning the original tree at the marks present at
time $q$. We also define the mass of  the tree $\ct_q$:
\[
\sigma_q=m^\ct(\ct_q).
\]
According to \cite{adv:plcrt}, $\ct_q$ is distributed under
$\N_\infty$ as a L\'evy tree with branching mechanism
$$\psi_q(u)=\psi(u+q)-\psi(q)=\alpha u^2+2\alpha qu.$$
We will denote by $\N^{\psi_q}$ the distribution
of $\ct_q$ under $\N$. Moreover, thanks to Girsanov formula (\cite{ad:ctvmp}, Lemma 6.2),
we have, for every nonnegative Borel function $F$
\begin{equation}
\label{eq:girsanov}
\N^{\psi_q}[F(\ct)]=\N[F(\ct_q)]
=\N\left[F(\ct)\expp{-\alpha q^2\sigma}\right].
\end{equation}

With the same abuse of notation as for the spanned subtree, we write for
every $t\in\R_+$, 
$\theta(t)$ instead of
$\theta(p(t))$. 

\subsection{Discrete trees}

We recall here the definition of a discrete ordered rooted tree
according to Neveu's formalism \cite{n:apgw}.

We consider $\cu=\displaystyle \bigcup_{n=0}^{+\infty}(\N^*)^n$ the set
of finite sequences of positive integers. The empty sequence
$\emptyset$ belongs to $\cu$. If $u,v\in \cu$, we denote by $uv$ the
sequence obtained by juxtaposing the sequences $u$ and $v$.

A discrete ordered rooted tree $T$ is a subset of $\cu$ satisfying the three following
properties
\begin{itemize}
\item $\emptyset\in T$. $\emptyset$ is called the root of $T$.
\item For every $u\in \cu$ and $i\in \N^*$, if $ui\in T$ then
  $u\in T$.
\item For every $u\in T$, there exists an integer $k_u(T)$ such
  that
$$ui\in T\iff 1\le i\le k_u(T).$$
\end{itemize}
The integer $k_u(T)$ is the number of offsprings of the vertex
$u$. The leaves of the tree are the $u\in T$ such that $k_u(T)=0$. We will consider here only binary trees i.e, discrete trees such
that $k_u(T)=0$ or $2$.

We can add edge lengths to a discrete tree by considering weighted
trees. A weighted tree is defined by a discrete ordered rooted tree
$T$ and a weight $h_u\in[0,+\infty)$ for every $u\in T$. The
  elements $u\in T$ must be viewed as the edges of the tree and $h_u$
  is the length of the edge $u$. Obviously, such a weighted tree can
  be viewed as a real tree and we will always make the confusion
  between a discrete weighted tree and the associated real tree.

\section{Janson's random variable}
\label{sec:janson}

Let $\ct$ be a compact real tree and let $M$ be a mark process on $\ct$.
We set
\[
\Theta=\int_\ct\theta(x)m^\ct(dx).
\]
Remark that this can be re-written using the coding by 
$\Theta=\int_0^\sigma \theta(s)ds$.

Using the tree-valued process $(\ct_q,q\ge 0)$, we can give another
expression for $\Theta$.  Let $(\theta_i, i\in \ci)$ be the set of jumping times of
$(\sigma_q, q\geq 0)$. We set:
\begin{equation}
   \label{eq:def-Ti}
\ct^i=\{x\in \ct; \theta(x)=\theta_i\}  \quad\text{and}\quad 
\sigma^i=m^\ct(\ct^i)=\sigma_{\theta_i-}-\sigma_{\theta_i}.
\end{equation}
According   to   \cite{ad:ctvmp},    we   have   that   $\M^{\ct}_\infty
$-a.s. $\ct^i$ is a real tree for all $i\in \ci$. 
Then the following result is straightforward as by definition
$\Theta=\sum_{i\in \ci} \theta_i \sigma^i$ and 
$\sigma_q=\sum_{\theta_i\geq q} \sigma^i$. 

\begin{prop}\label{prop:expression-theta}
We have  $\M^\ct_\infty$-a.s.:
\[
\Theta=\int_0^{+\infty}\sigma_q\, dq.
\]
\end{prop}

The main result of this section is the following theorem that identifies
$\Theta$ as Janson's random variable whose distribution is
characterized by its moments.

\begin{theo}\label{thm:main}
For every positive integer $r$, we have
\[
\M_\infty^\ct[\Theta^r]=\frac{r!}{(2\alpha)^r}\, \int_{\ct^r}\frac{m^\ct(dx_1)\ldots
  m^\ct(dx_r)}{\prod_{i=1}^r\ell(\ct(x_1,\ldots,x_i))}\cdot
\]
\end{theo}

\begin{proof} 
Using the expression of Proposition \ref{prop:expression-theta} for
$\Theta$, we have
\begin{align*}
\M_\infty^\ct[\Theta^r] 
& = r!\ \M_\infty^\ct\left[\int_{0\le
    q_1<q_2<\cdots<q_r}dq_1\ldots dq_r\, \sigma_{q_1}\ldots
  \sigma_{q_r}\right]\\
& =r!\, \M_\infty^\ct\left[\int_{0\le
    q_1<q_2<\cdots<q_r}dq_1\ldots dq_r \prod_{k=1}^r\int_\ct
  m^\ct(dx_k)\ind_{\{x_k\in\ct_{q_k}\}}\right]\\
& =r!\, \int_{\ct^r}m^\ct(dx_1)\ldots m^\ct(dx_r)\int_{0\le
    q_1<q_2<\cdots<q_r}dq_1\ldots dq_r\, \M_\infty^\ct[x_1\in\ct_{q_1},\ldots,x_r\in\ct_{q_r}].
\end{align*}

To evaluate the probability that appears in the last equation, let us
remark that, if $y\in \ct_{q}$, then $y\in \ct_{q'}$ for every
$q'<q$. Therefore, we have
\begin{multline*}
\M_\infty^\ct[x_1\in\ct_{q_1},\ldots,x_r\in\ct_{q_r}]\\
=\M_\infty^\ct[x_2\in\ct_{q_2},\ldots,x_r\in\ct_{q_r}\bigm|
x_1\in\ct_{q_1},\ldots
,x_r\in\ct_{q_1}]\M_\infty^\ct[x_1\in\ct_{q_1},\ldots
  ,x_r\in\ct_{q_1}].
\end{multline*}

On one hand, we have
\begin{align*}
\M_\infty^\ct[x_1\in\ct_{q_1},\ldots
  ,x_r\in\ct_{q_1}] & =\M_\infty^\ct\left[
  M(\ct(x_1,\ldots,x_r)\times[0,q_1])=0\right]\\
& =\exp\bigl(-2\alpha q_1\ell(\ct(x_1,\ldots,x_r))\bigr).
\end{align*}

On the other hand, by standard properties of Poisson point measures,
we have
$$\M_\infty^\ct[x_2\in\ct_{q_2},\ldots,x_r\in\ct_{q_r}\bigm|
x_1\in\ct_{q_1},\ldots
,x_r\in\ct_{q_1}]=\M_\infty^\ct[x_2\in\ct_{q_2-q_1},\ldots,x_r\in\ct_{q_r-q_1}].$$

We finally obtain by induction, with the convention $q_0=0$:
\[
\M_\infty^\ct[x_1\in\ct_{q_1},\ldots,x_r\in\ct_{q_r}]
=\prod_{k=1}^r\expp{-2\alpha (q_k-q_{k-1})\ell(\ct(x_k,\ldots,
x_r))}.
\]

Plugging this expression in the integral gives, after an obvious
change of variables
\begin{align*} 
\M_\infty^\ct[\Theta^r] 
& =r! \int_{\ct^r}m^\ct(dx_1)\ldots m^\ct(dx_r)\int_{0\le
    q_1<\cdots<q_r}\!\!\! dq_1\ldots dq_r\prod_{k=1}^r\expp{-2\alpha
    (q_k-q_{k-1})\ell(\ct(x_k,\ldots, 
x_r))}\\
& =r! \int_{\ct^r}m^\ct(dx_1)\ldots m^\ct(dx_r)\prod_{k=1}^r\int
_0^{+\infty}da_k\expp{-2\alpha a_k\ell(\ct(x_k,\ldots, 
x_r))}\\
& =\frac{r!}{(2\alpha)^r} \int_{\ct^r}\frac{m^\ct(dx_1)\ldots
  m^\ct(dx_r)}{\prod_{k=1}^r\ell(\ct(x_k,\ldots, 
x_r))}\cdot
\end{align*}
\end{proof}
We can then deduce from the results of
\cite{j:rcrdrt}, that  for $\alpha=2$, under $\N_\infty ^{(1)}$,
$\Theta$ has  Rayleigh distribution. Using then scaling argument in $r$ and
$\alpha$ or directly Corollary \ref{cor:partilal-F-infity} in the
Appendix, we get the following result. 

\begin{cor}
   \label{prop:Laplace-theta}
For  all  $r>0$,  the  random variable
   $Z=\sqrt{\frac{2\alpha}{r}}\Theta$ is distributed under $\N_\infty
^{(r)}$ according to a  Rayleigh distribution with density 
$x\expp{-x^2/2}\ind_{\{x\ge 0\}}$.
\end{cor}

In particular, we have easily the first moments of $\Theta$:
\begin{equation}
   \label{eq:moment-int-theta}
\N_\infty ^{(r)}\left[\Theta\right]= \frac{1}{2}\sqrt{\frac{\pi r}{\alpha}}
\quad\text{and}\quad
\N_\infty ^{(r)}\left[\Theta^2\right]= \frac{r}{\alpha}\cdot
\end{equation}

\section{A.s. convergence}
\label{sec:convergence}

\subsection{Statement of the main result}

Let $r\ge 0$ and let $\ct$ be a tree distributed according to
$\N^{(r)}$. 
Let $(U_1,\ldots,U_n)$ be $n$  points  uniformly chosen at random on
$[0,r]$, independent of $\ct$. We denote by $T_n$ the random tree
spanned by these $n$ points i.e.
$$T_n=\ct(U_1,\ldots,U_n)$$
viewed  as a  discrete ordered  weighted  tree.  Notice  that $T_n$  has
$2n-1$ edges. Let $(h_1, \ldots,  h_{2n-1})$ be the lengths of the edges
given in lexicographic order.  We consider the total length of $T_n$:
\[
L_n=\ell(T_n)=\sum_{k=1}^{2n-1} h_k.
\]

We define $m_n$ as the first branching point of $T_n$, i.e.
\begin{equation}
   \label{eq:def-mn}
\bigcap_{k=1}^n\lb\emptyset,p(U_k)\rb=\lb\emptyset,m_n\rb
\end{equation}
and we consider the length of the edge of $T_n$ attached to the root
\begin{equation}
   \label{eq:def-h0n}
h_{\emptyset,n}:=d(\emptyset,m_n)=\ell(\lb \emptyset,m_n\rb)=h_1.
\end{equation}

Let $T_n^*$ be the
sub-tree of $T_n$  where
we   remove  the   edge  $\lb\emptyset,m_n\lb$:
\[
T_n^*=T_n\setminus \lb\emptyset,m_n\lb,
\]
and $L_n^*$ its  total length \textit{i.e.} $L_n^*=L_n-h_{\emptyset,n}$.

We set $\theta(x-)=\inf \{\theta(y), y\in\lb\emptyset,x\lb\}$  and   $X_n^*$ the number
of records on the tree $T_n^*$:
\[
X_n^*=\sum_{x\in T^*_n} \ind_{\{\theta(x-)>\theta(x)\}},
\]

\begin{rem}
The introduction of the tree $T_n^*$ is motivated by the fact that the
number
$$\sum_{x\in T_n} \ind_{\{\theta(x-)>\theta(x)\}}$$
of records  on the whole  tree is $\N_\infty$-a.e.  infinite.  Moreover,
$X_n^*+1$ represents the number of cuts that appears on the reduced tree
$T_n$ until  a mark  appears on  the branch attached  to the  root which
reduces the  tree to a trivial one  consisting of the root  and a single
branch attached to it. Hence it is the analogue of the discrete quantity
$X_n$ and is the right quantity to be studied.
\end{rem}

We can then state the main result of this section which will be
proven in Section \ref{sec:records}.
\begin{theo}
   \label{theo:CV-X_n}
We have that, for all $r>0$,  $\N^{(r)}_\infty $-a.s.:
\[
\lim_{n\rightarrow+\infty }
\frac{X_n^*}{\sqrt{2n}}=\sqrt{\frac{\alpha}{2r}} \Theta=Z. 
\]
\end{theo}

\begin{rem}
   \label{rem:janson}
   Notice  that  the binary  tree  $T_n$  has  $2n-1$ vertices;  and  it
   corresponds to  a critical  Galton-Watson tree with  reproduction law
   taking values in  $\{0,2\}$ and with variance 1  conditionally on its
   number of edges being  $2n-1$. This and Theorem 1.6 in \cite{j:rcrdrt}
   for $\alpha=1/2$  and $r=1$,  imply that the  number of  edges with
   more than one cut is of order less that $\sqrt{n}$.

We deduce from Theorem \ref{theo:CV-X_n} and Corollary \ref{cor:cv-Ln}
that for all $r>0$,  $\N^{(r)}_\infty $-a.s.:
\[
\lim_{n\rightarrow+\infty }
\frac{X_n^*}{L_n}=\alpha \frac{\Theta}{\sigma}\cdot
\]
In the left  hand-side, we have the average of the  number of records on
$T^*_n$ (as  $\ell(T^*_n)$ is  of the  same order as  $L_n$) and  in the
right  hand-side, the  ratio  $\Theta/\sigma$ appears  as  the value  of
$\theta(U)$ for  a leaf chosen  uniformly according to the  normalized mass measure
$m^\ct/\sigma$  and  $\alpha$ is  a constant
related to the branching mechanism. This  result is then natural as
intuitively the normalized mass measure is the weak limit of 
the  normalized length  measure on  $T_n$. 
\end{rem}

\subsection{Other a.s. convergence results}\label{sec:another-repr}

Recall the definition \reff{eq:def-pruned} of the pruned sub-tree $\ct_q$.
Let $\T$ be the set of trees with their mass measure (see
\cite{adh:GHP}). 
We define the
backward filtration $\cg=(\cg_q,q\ge 0)$ with $\cg_q=\sigma(\ct_r, r\geq
q)$.
 Following  \cite{adh:etiltvp},  we get  that
the random measure:
\[
\cn(d\ct',dq)=\sum_{i\in \ci} \delta_{\ct^i,\theta_i}(d\ct',dq)
\]
is under $\N_\infty$ a point measure on $\T\times \R$ with intensity:
\[
\ind_{\{q>0\}} 2\alpha \sigma_q\, \N^q\left[d\ct' \right] \; dq.
\]
This means that for  every non-negative predictable process $(Y(\ct', q), q\in
\R_+, \ct'\in \T)$ with respect to the backward filtration $\cg$,
\begin{equation}
   \label{eq:intensite}
\N_\infty\left[\int Y(\ct', q)\cn(d\ct',dq)\right]=\N_\infty
\left[\int \cy_q\;\ind_{\{q>0\}} 2\alpha \sigma_q\; dq\right],
\end{equation}
where  $(\cy_q=\int Y(\ct',  q) \N^q[d\ct'],q\in  \R_+)$  is predictable
with   respect  to  the   backward  filtration   $\cg$.   We   refer  to
\cite{dvj:itpp1,dvj:itpp2}  for  the   general  theory  of  random  point
measures.

Recall $\sigma^i=m^\ct (\ct^i)$. 

\begin{prop}
\label{prop:cv-sqi}
We have $\N_\infty$-a.e.:
\[
\lim_{n\to+\infty}\frac{1}{\sqrt
  n}\sum_{i\in\ci}\ind_{\{\sigma^i\geq 1/n\}}=
2\sqrt\frac{\alpha}{\pi}\Theta= \sqrt{\frac{2\sigma}{\pi}}\; Z. 
\]
\end{prop}

\begin{proof}
Let $K>0$ be large. We consider the $\cg$-stopping time $\tau_K=\inf\{q;
\sigma_q <K/2\alpha\}$. We define for every $\theta>0$ and every positive integer $n$,
\[
Q_n(\theta)=\sum_{i\in\ci}\ind_{\{
  \sigma^i\geq 1/n\}}\ind_{\{\theta_i> \theta\}}.  
\]
We have $Q_n(\tau_K)=\sum_{i\in I} \ind_{\{
  \sigma^i\geq 1/n\}}\ind_{\{\sigma_{\theta_i+}<K/2\alpha\}} $ so  that:
\begin{align*}
\N_\infty \left[Q_n(\tau_K)\right]
&= \N_\infty\left[\int_{ \tau_K}
  ^{+\infty}dq\; \sigma_q\N\left[\ind_{\{ \sigma\geq 1/n\}}\expp{-\alpha 
  q^2\sigma}\right]\right]\\
&\leq \N_\infty\left[\int_0
  ^{+\infty}dq\; \min\left(\sigma_q, \frac{K}{2\alpha}\right)
  \N\left[\ind_{\{ \sigma\geq 1/n\}}\expp{-\alpha  
  q^2\sigma}\right]\right]\\
&=  \int_0^{+\infty} dq\; \N\left[\min\left(\sigma,
    \frac{K}{2\alpha}\right) \expp{-\alpha q^2 \sigma} \right]
 \N\left[\ind_{\{ \sigma\geq 1/n\}}\expp{-\alpha  
  q^2\sigma}\right]\\
&=\inv{4\alpha{\pi}} 
\int_0^{+\infty}dq\; \int_0^{+\infty }  \frac{du}{u^{3/2}} \min\left(u,
    \frac{K}{2\alpha}\right) \expp{-\alpha  
    q^2 u}
\int_{1/n}^{+\infty}\frac{dr}{r^{3/2}}\expp{-\alpha  
    q^2 r}\\
&=\inv{8\alpha^{3/2} \sqrt{\pi}} 
\int_{\R_+^2}   \frac{du}{u^{3/2}}\, \frac{dr}{r^{3/2}}\,   \min\left(u,
    \frac{K}{2\alpha}\right) \inv{\sqrt{u+r}}\ind_{\{r>1/n\}},
\end{align*}
where  we  used \reff{eq:intensite}  for  the  first equality,  Girsanov
formula \reff{eq:girsanov},  and the density  \reff{eq:densite-s} of the
distribution of $\sigma$ under $\N$.
Elementary computations yields there exists a finite constant $c$ which
depends on $K$ but not on $n$ such that:
\begin{equation}
   \label{eq:majo-intensite}
\N_\infty \left[Q_n(\tau_K)\right]
= \N_\infty\left[\int_{ \tau_K}
  ^{+\infty}dq\; \sigma_q\N\left[\ind_{\{ \sigma\geq 1/n\}}\expp{-\alpha 
  q^2\sigma}\right]\right]\\
\leq  c\sqrt{n}(1+\log(n)).
\end{equation}
Classical results on random point measures imply that the process
$(N_n(\theta\vee \tau_K),\theta\geq  0)$, with:
\[
N_n(\theta)=Q_n(\theta)-2\alpha\int_\theta^{+\infty}
dq\;\sigma_q\N\left[\ind_{\{\sigma\geq 1/n\}}\expp{-\alpha
  q^2\sigma}\right]
\]
is a backward martingale with respect to $\cg$. Moreover, since
$(Q_n(\theta),\theta\ge 0)$ is a pure jump process with jumps of size
1, the process $(M_n(\theta\vee \tau_K),\theta\geq  0)$, with:
\[
M_n(\theta)=N_n(\theta)^2-2\alpha\int_\theta^{+\infty}dq\;\sigma_q
\N\left[\ind_{\{\sigma\geq 1/{n}\}}\expp{-\alpha  
  q^2\sigma}\right]
\]
is also a backward  martingale with respect to $\cg$. 
Using \reff{eq:majo-intensite}, we get that $   \N_\infty\left[
   \left(N_{n^4}(\tau_K)/n^2\right)^2\right]$ is less than a constant times
   $n^{-3/2}$; 
   therefore 
\[
\sum_{n=1}^{+\infty}\left(\frac{N_{n^4}(\tau_K)}{n^2}\right)^2
\]
is finite in $L^1(\N_\infty )$ and thus is $\N_\infty$-a.e. finite. This
implies that $\N_\infty$-a.e.:
\[
\lim_{n\rightarrow+\infty }
\frac{N_{n^4}(\tau_K)}{n^2}=0.
\]
Moreover, we have by monotone convergence:
\begin{align*}
\frac{2\alpha}{\sqrt{n}}\int_{\tau_K}^{+\infty}dq\;
\sigma_q\N\left[\ind_{\{\sigma\geq 1/{n}\}}\expp{-\alpha 
 q^2\sigma}\right] 
&=2\alpha\int_{\tau_K}^{+\infty}dq\; \sigma_q
\int_1^{+\infty}\frac{dr}{2\sqrt{\alpha\pi}r^{3/2}}\expp{-\alpha 
  q^2 \frac{r}{n}}\\ 
&\; \xrightarrow[n\rightarrow \infty ]{\N_\infty \text{-a.e.}} \; 
2\sqrt\frac{\alpha}{\pi}\int_{\tau_K}^{+\infty}dq\sigma_q.
\end{align*}
We get that the sequence $(Q_{n^4}(\tau_K)/n^2, n\geq 1)$ converges
$\N_\infty$-a.e. toward 
$2\sqrt\frac{\alpha}{\pi}\int_{\tau_K} ^{+\infty}dq\; \sigma_q$. Since
$(Q_n(\theta), n\ge 1)$ is 
non-decreasing, we deduce that $\N_\infty $-a.e.:
\[
\lim_{n\rightarrow +\infty } \frac{Q_n(\tau_K)}{\sqrt
  n}=2\sqrt\frac{\alpha}{\pi} \int_{\tau_K}^{+\infty}dq\; \sigma_q.
\]
Since $\sigma $ is finite $\N_\infty $-a.e., we get that $\N_\infty
$-a.e. $\tau_K=0$ for
$K$ large enough. This gives the result. 
\end{proof}

\begin{prop} We have $\N_\infty$-a.e.:
\label{prop:cv-sqi2}
\[
\lim_{n\to+\infty}\sqrt{n}\sum_{i\in\ci}\sigma^i\ind_{\{\sigma^i\le
1/    n\}}=2\sqrt{\frac{\alpha}{\pi}}\Theta=
  \sqrt{\frac{2\sigma}{\pi}}\; Z. 
\]
\end{prop}

\begin{proof}
  The proof is very similar to the proof of Proposition \ref{prop:cv-sqi}. We set:
\[
Q_n(\theta)=\sum_{i\in\ci}\sigma^i\ind_{\{\sigma^i\le
    1/n\}}\ind_{\{\theta_i\ge \theta\}}.
\]
Mimicking the  proof of Proposition \ref{prop:cv-sqi}, we  have for some
finite constant $c$ which depends on $K$ but not on $n$:
\begin{align*}
\N_\infty \left[Q_n(\tau_K)\right]
&= \N_\infty\left[\int_{ \tau_K}
  ^{+\infty}dq\; \sigma_q\N\left[\sigma \ind_{\{
      \sigma\leq 1/n\}}\expp{-\alpha  
  q^2\sigma}\right]\right]\\
&\leq \inv{8\alpha^{3/2} \sqrt{\pi}} 
\int_{\R_+^2}   \frac{du}{u^{3/2}}\, \frac{dr}{r^{3/2}}\,   \min\left(u,
    \frac{K}{2\alpha}\right) \inv{\sqrt{u+r}}\; r\ind_{\{r\leq 1/n\}}\\
&\leq c n^{-1/2}(1+\log(n))
<+\infty ,
\end{align*}
as well as:
\begin{multline*}
  \N_\infty \left[ \int_{\tau_K}^{+\infty}dq\;\sigma_q
\N\left[\sigma^2 \ind_{\{\sigma\leq 1/{n}\}}\expp{-\alpha  
  q^2\sigma}\right]   \right]\\ 
\begin{aligned}
&\leq \inv{8\alpha^{3/2} \sqrt{\pi}} 
\int_{\R_+^2}   \frac{du}{u^{3/2}}\, \frac{dr}{r^{3/2}}\,   \min\left(u,
    \frac{K}{2\alpha}\right) \inv{\sqrt{u+r}}\; r^2\ind_{\{r\leq 1/n\}}\\
&\leq  c n^{-3/2}(1+\log(n)).
\end{aligned}
\end{multline*}
Classical results on random point measures imply that
the processes 
$(N_n(\theta\vee \tau_K),\theta\geq  0)$ and $(M_n(\theta\vee \tau_K),\theta\geq  0)$, with:
\begin{align*}
N_n(\theta)
&=Q_n(\theta)-2\alpha\int_\theta^{+\infty}
dq\;\sigma_q\N\left[\sigma \ind_{\{\sigma\leq 1/n\}}\expp{-\alpha
  q^2\sigma}\right]\\
M_n(\theta)
&=N_n(\theta)^2-2\alpha\int_\theta^{+\infty}dq\;\sigma_q
\N\left[\sigma^2 \ind_{\{\sigma\leq 1/{n}\}}\expp{-\alpha  
  q^2\sigma}\right]   
\end{align*}
are  backward martingales with respect to $\cg$. 
We get that $   \N_\infty\left[
   \left(n^2N_{n^4}(\tau_K)\right)^2\right]$ is less than a constant times
   $n^{-3/2}$. Following  the proof of  Proposition \ref{prop:cv-sqi}, we
   deduce that $\N_\infty$-a.e. $\lim_{n\rightarrow+\infty }
n^2N_{n^4}(\tau_K)=0$. 
Furthermore, we have:
\begin{align*}
2\alpha\sqrt n \int_{\tau_K}^{+\infty}dq\;
\sigma_q\N\left[\sigma\ind_{\{\sigma\le 
   1/n\}}\expp{-\alpha q^2\sigma}\right] 
& =2\alpha\sqrt n
\int_{\tau_K} ^{+\infty}dq\;
\sigma_q\int_0^\frac{1}{n}\frac{dr}{2\sqrt{\alpha\pi 
  r}}\expp{-\alpha q^2 r}\\
& = 2\alpha
\int_{\tau_K}^{+\infty}dq\; \sigma_q\int_0^1\frac{dr}{2\sqrt{\alpha\pi
  r}}\expp{-\alpha q^2 \frac{r}{n}}\\
& \to 2\sqrt{\frac{\alpha}{\pi}}\int_{\tau_K}^{+\infty}dq\; \sigma_q. 
\end{align*}
We conclude the proof as in the proof of  Proposition
\ref{prop:cv-sqi}. 
\end{proof}

\subsection{The record process on the real half-line}

We consider here the half-line $[0,+\infty)$ instead of a real-tree
  $\ct$ (the half-line is in fact a real tree that we supposed rooted
  at $0$). We define the mark process $M$ under $\M_a$ (we omit the
  $\ct=[0,+\infty)$ in the notation), it is a Poisson point measure on
    $[0,+\infty)^2$ with intensity $2\alpha\ind_{\{x\ge 0,\ 0\le q\le
        a\}}dx\, dq$ and we set for every $x\ge 0$
\[
\theta(x)=\min(a,\inf\{q_i; x_i\leq x\})\quad\text{and}\quad
X(x)=X(0)+ \sum_{0<y\leq x} \ind_{\{\theta(y-)> \theta(y)\}}.
\]

\begin{rem}
Let us denote by $1\ge x_1>x_2>\cdots$ the jumping times of the
process $(\theta(x),0\le x\le 1)$ under $\M_\infty$. By standard
arguments on Poisson point measure, the random variable $x_1$ is
uniformly distributed on $[0,1]$. Conditionally given $x_1$, the
random variable $x_2$ is uniformly distributed on $[0,x_1]$ and so
on. We are thus considering the standard stick breaking scheme and the
random vector $(1-x_1,x_1-x_2,\ldots)$ is distributed according to the
Poisson-Dirichlet distribution with parameter $(0,1)$.
\end{rem}

For  fixed $x$,  $\theta(x)$ represents  the first  time a  mark arrives
between  $x$  and  $0$  (if  it  arrives before  time  $a$  that  is  if
$\theta(x)<a$);  and  $X(x)-X(0)$  denotes  the number  of  (decreasing)
records  of the  process $(\theta(u),\,  u\in  [0,x])$. It  is also  the
number of  cuts that  appear between $x$  and $0$ in  the stick-breaking
scheme before time $a$.

By  construction $\theta$  and $(\theta,X)$  are Markov  processes. 
Notice  that  $\theta$  is non-increasing and $X$ is  non-decreasing,  and
$\M_\infty$-a.s. $X(x)=+\infty $ for every $x> 0$.

As most  of our further proofs will be based on martingale arguments, let
us first compute the infinitesimal generator of the former Markov
processes. Notice first that  $\inf\{q_i; x_i\leq x\}$
is distributed under $\M_\infty$ as an exponential random variable with parameter
$2\alpha x$.  Let $g$ be a bounded measurable function
defined on $[0,+\infty ]$. For every  $q\in [0, +\infty ]$ and $x>0$,
we have
\[
\M_{q}[g(\theta(x))]
=\M_\infty[g(\min(q,Y_x))]
= \expp{-2\alpha q x} g(q) +\int_0^qg(u) \; 2\alpha x\expp{-2\alpha
  xu}\; du, 
\]
where $Y_x$ is exponentially distributed with parameter $2\alpha x$.
Notice that if $g$ belongs to $\cc^1(\R^+)$ with $g'$ bounded on $\R_+$,
we have by an obvious integration by parts that,  for $q\in [0, +\infty ]$ and $x>0$,
\[
\M_{q}[g(\theta(x))]
= g(0) + \int_0^q g'(u)    \; \expp{-2\alpha xu}\; du .
\]
We can then compute the infinitesimal generator of $\theta$ denoted by
$\cl$. Let $g$ be a bounded measurable function defined on
$[0,+\infty ]$ such that $g-g(+\infty )$ is integrable with respect to
the Lebesgue measure on $\R^+$. For $q\in [0, +\infty  ]$, we have:
\begin{align*}
\cl(g)(q) & =\lim_{x\rightarrow 0} \frac{\M_{q}[g(\theta(x))] -g(q)}{x}\\
& = \lim_{x\rightarrow 0} -g(q) \frac{1-\expp{-2\alpha qx}}{x} +
\int_0^q\!\! 2\alpha g(u)
\expp{-2\alpha xu}\; du \\
& = 2\alpha \int_0^q \!(g(u) -g(q))\;  du.
\end{align*}
Therefore, we get that the process $M^g=(M_x^g,
x\geq 0)$ is a martingale under $\M_q$, where $M^g$ is defined by:
\begin{equation}
   \label{eq:def-Mg}
M_x^g=g(\theta(x)) + 2\alpha\int_0^x dy \int_0^{\theta(y)}
\Big(g(\theta(u)) -  g(y)\Big)\; du.
\end{equation}

\begin{rem}
   \label{rem:gC1}
If furthermore $g$ belongs to  $\cc^1(\R^+)$ and if $x\mapsto xg'(x) $ is integrable with respect to
the Lebesgue measure on $\R^+$, then we have for $q\in [0, +\infty  ]$:
\[
\cl(g)(q)= - 2\alpha \int_0^q xg'(x)\; dx.
\]
\end{rem}

Similarly, we can also  compute the infinitesimal generator of $(\theta,
X)$, which we  still denote by $\cl$. This quantity  is of interest only
for $\theta(0)$  finite. Let  $g$ be a bounded measurable function defined on
$\R^+\times \N$.
For $(q,k)\in \R^+\times \N$, we denote by $\M_{(q,k)}$ the law of the
process $(\theta,X)$ starting from $(q,k)$.
 Standard computations on birth and death processes yield
that for $(q,k)\in \R^+\times \N$:
\begin{align*}
\cl(g)(q,k)
&=\lim_{x\rightarrow 0} \frac{\M_{(q,k)}[g(\theta(x),X(x))] -g(q,k)}{x}\\
&= \lim_{x\rightarrow 0} -g(q,k) \frac{1-\expp{-2\alpha qx}}{x}  +
\int_0^q2\alpha g(u,k+1) \;
\expp{-2\alpha xu}\; du +o(1)\\
&= 2\alpha\int_0^q (g(u,k+1) -g(q,k))\;  du.
\end{align*}
In that case, we get that the process $M^g=(M_x^g,
x\geq 0)$ defined by:
\begin{equation}
   \label{eq:def-MgX}
M_x^g=g(\theta(x), X(x)) - 2\alpha\int_0^x dy \int_0^{\theta(y)} \bigg(g(u,X(y)+1)
-g(\theta(y),X(y))\bigg) \; du ,
\end{equation}
is a bounded martingale under $\M_{(q,k)}$.

Finally, let us exhibit some martingales associated with the process
$X$ which show that this process can be viewed as a Poisson process
with stochastic intensity $2\alpha \theta(u)du$.
Let $n\in \N$. Taking $g(q,k)=k\wedge n$ in \reff{eq:def-MgX}, we deduce that the
process $N^{(n)}=(N^{(n)}_x,x\geq 0)$  defined for 
$x\geq 0$ by:
\[
N^{(n)}_x=X(x)\wedge n -2\alpha\int_0^x  \theta(u)\ind_{\{X(u)<n\}}\; du
\]
is a bounded martingale under $\M_{(q,k)}$ (for $q<+\infty $). Notice that for
$(q,k)\in \R^+\times \N$, we have:
\begin{align*}
\M_{(q,k)}[|N^{(n)}_x|]
&\leq  
\M_{(q,k)}[X(x)\wedge n] + 2\alpha\int_0^x  \E_{(q,k)}[\theta(u)] \; du\\
&= k\wedge n+ 2\alpha\int_0^x   \E_{(q,k)}[\theta(u)\ind_{\{X(u)<n\}}]\; du +
2\alpha\int_0^x  \E_{(q,k)}[\theta(u)] \; du \\
&\leq  k+4\alpha qx,
\end{align*}
where  we used  that $X$  is non-negative  in the  first  equality, that
$N^{(n)}$  is  a  martingale  in   the  second one,  and  that  $\theta$  is
non-increasing in the last one. As $(N^{(n)}, n\in \N)$ converges
a.s. to the process $N=(N_x, x\geq 0)$ defined for $x\in \R^+$ by:
\begin{equation}
   \label{eq:X-mart}
N_x=X(x) -2\alpha\int_0^x  \theta(u)\; du, 
\end{equation}
we deduce that $N$ is a martingale under $\M_{(q,k)}$
for every $(q,k)\in \R^+\times \N$.

By  taking $g(q, k)=k^2$ in \reff{eq:def-MgX}  and using  elementary stochastic  calculus and
similar arguments as  above, we also get that the  process $M=(M_x,x\geq 0)$
defined for $x\geq 0$ by:
\begin{equation}
   \label{eq:X2-mart}
M_x=N_x^2 -2\alpha\int_0^x \theta(u)\; du 
\end{equation}
is  a martingale under $\M_{(q,k)}$ for every $(q,k)\in
\R^+\times \N$. 

\subsection{Sub-tree spanned by $n$ leaves}\label{sec:subtree}

We recall here some properties of the sub-tree spanned by $n$ leaves
uniformly chosen.

We first recall the density of $(h_1, \ldots, h_{2n-1})$ under $\N^{(r)}$, see
\cite{a:crt3} or \cite{p:csp} (Theorem 7.9), see also
\cite{dlg:pfalt}. We denote by $L_n$ the total length of $T_n$:
\[
L_n=\sum_{k=1}^{2n-1}h_k.
\]

\begin{lem}
   \label{lem:densite-h}
Under $\N^{(r)}$, $(h_1,
\ldots, h_{2n-1})$ has density:
\[
f_n^{(r)} (h_1, \ldots, h_{2n-1})= 2\,  \frac{(2n-2)!}{(n-1)!}\,
\frac{\alpha^n}{r^n} \, 
L_n\expp{-\alpha L_n^2/r} 
\ind_{\{h_1>0, \ldots, h_{2n-1}>0\}}.
\]
The random variable $L_n^2$, is distributed
   under $\N^{(r)} $ as $r \Gamma_n/\alpha$ where $\Gamma_n$ is a
   $\gamma(1,n)$ random variable with density $ \ind_{\{x>0\}}\, x^{n-1}
   \expp{-x}/(n-1)!$. 
\end{lem}

\begin{cor}
\label{cor:cv-Ln}
 We have that $\N^{(r)}$-a.s. 
\[
\lim_{n\rightarrow+\infty } L_n/\sqrt{n}=\sqrt{r/\alpha}.
\]
\end{cor}

\begin{proof}
Using Lemma \ref{lem:densite-h}, we compute
\[
   \N^{(r)}\left[\sum_{n=1}^{+\infty} \left(\frac{L_n^2}{n} -
     \frac{r}{\alpha} \right)^4 \right]
=\frac{r}{\alpha} \sum_{n=1}^{+\infty} \E\left[\left(\frac{\Gamma_n}{n} -
     1 \right)^4\right]
= \frac{r}{\alpha} \sum_{n=1}^ {+\infty }
\inv{n^2}\left(3+\inv{n}\right)<+\infty . 
\]
This implies that $\N^{(r)}$-a.s. $\sum_{n=1}^{+\infty} \left(\frac{L_n^2}{n} -
     \frac{r}{\alpha} \right)^4$ is finite which proves the corollary.
\end{proof}

We end this section by studying the edge attached to the root defined in
\reff{eq:def-mn} whose length is denoted $h_{\emptyset,n}$, see \reff{eq:def-h0n}.

\begin{prop}
   \label{prop:h0}
   The  sequence  $(\sqrt{n}h_{\emptyset,n},   n\geq  1)$  converges  in
   distribution under $\N^{(r)}$ to $\sqrt{r/\alpha} \; E_1/2$,  where $E_1$ is an exponential
   random variable with mean 1.
\end{prop}
\begin{proof}

Let $k\in (-1,+\infty )$. 
    We set $H_k=(\alpha/r)^{k/2}\N^{(r)}[h_{\emptyset,n}^k]$. We have
    using Lemma \ref{lem:densite-h},
\[
H_k
 = 2\frac{(2n-2)!}{(n-1)!}\frac{\alpha^{n+k/2}}{r^{n+k/2}}
 \int_{\R_+^{2n-1}}dh_1\ldots  dh_{2n-1} 
 h_1^k
\; L_n\expp{-\alpha L_n^2/r}. 
\]
Consider the change of variables:
\[
u_1=\sqrt{\frac{\alpha}{r}}h_1, \;
\cdots,\; 
u_{2n-2}=\sqrt{\frac{\alpha}{r}}h_{2n-2},\; 
x=\sqrt{\frac{\alpha}{r}}L_n, 
\]
with Jacobian equal to $\left(\frac{\alpha}{r}\right)^{n-\frac{1}{2}}$.We get:
\begin{align*}
H_k & =2\frac{(2n-2)!}{(n-1)!}\frac{\alpha^{n+k/2}}{r^{n+k/2}}
 \int_{\R_+^{2n-1}}\left(\frac{r}{\alpha}\right)^{k/2}u_1^k\left(\frac{r}{\alpha}\right)^{1/2}x\expp{-\alpha
   x^2/r}\\
&\hspace{6cm}\ind_{\{u_1+\cdots+u_{2n-2}\le
 x\}}\left(\frac{r}{\alpha}\right)^{n-\frac{1}{2}} du_1\cdots
du_{2n-2}\, dx\\
& =2\frac{(2n-2)!}{(n-1)!}\int\int_{\R_+^2}du_1\, dx\, \ind_{\{u_1\le x\}}\;
u_1^kx\expp{-x^2}\int\int_{\R_+^{2n-3}}du_2\ldots  
du_{2n-2}\ind_{\{u_2+\cdots +u_{2n-2}\le x-u_1\}}\\
& =2\frac{(2n-2)!}{(n-1)!}\frac{1}{(2n-3)!}\int_0^{+\infty}dx\, 
x\expp{-x^2}\int_0^xdh\, 
h^k(x-h)^{2n-3}. 
\end{align*}
Set $y=x^2$, to get:
\begin{align*}
H_k 
& =2\; \frac{(2n-2)!}{(n-1)!}\; \frac{1}{(2n-3)!}\beta(k+1,2n-2)
\int_0^{+\infty}dx \;x^{2n+k-1} \expp{-x^2} \\
& =\frac{(2n-2)!}{(n-1)!} \;  \frac{1}{(2n-3)!}\beta(k+1,2n-2)
\int_0^{+\infty}dy\; y^{n+\frac{k}{2}-1} \expp{-r} \\
&=\frac{(2n-2)!}{(n-1)!}\;\inv{(2n-3)!}\; \frac{
\Gamma(k+1) (2n-3)! }{\Gamma(2n+k-1)} \; \Gamma(n+\frac{k}{2})
\\
&=\frac{\Gamma(k+1)}{2^{k}}\; \frac{\Gamma(n- \frac{1}{2})}{\Gamma(n+
  \frac{k}{2}- \frac{1}{2})},
\end{align*}
where, for the last equality, we used twice the duplication formula:
\begin{equation}
   \label{eq:duplication}
\frac{\Gamma(2n-1)}{\Gamma(n)}=\frac{2^{2n-2} \Gamma(n-1/2)}{\sqrt{\pi} }\cdot
\end{equation}
We observe that $\lim_{n\rightarrow +\infty }
\N^{(r)}[n^{k/2}h_{\emptyset,n}^k] = \frac{k!}{2^k}
\left(\frac{r}{\alpha}\right)^{k/2}
=\E[(\sqrt{r}E_1/(2\sqrt{\alpha}))^k]$. This gives 
the result, as the exponential distribution is characterized by its
moments. 
\end{proof}
From the proof of Proposition \ref{prop:h0}, we also get the following
result. 
\begin{lem}
   \label{lem:moment-h0}
For all $k\in (-1, +\infty )$, we have, when $n$ goes to infinity:
\[
\N^{(r)}[h_{\emptyset,n}^k]=\left(\frac{r}{\alpha}\right)^{k/2}
\frac{\Gamma(k+1)}{2^{k}}\; \frac{\Gamma(n- \frac{1}{2})}{\Gamma(n+
  \frac{k}{2}- \frac{1}{2})} \sim (r/\alpha)^{k/2} n^{-k/2}2^{-k}\Gamma(k+1).
\]
\end{lem}

\subsection{Proof of Theorem \ref{theo:CV-X_n}}\label{sec:records}

We first want to show that, as for a standard Poisson process, the
record counting process on each branch behaves like its (stochastic)
intensity when the number of jumps tends to $+\infty$.

In other words, we set:
\[
 \Delta_n=   \frac{X_n^*}{\sqrt{n}} - \frac{2\alpha}{\sqrt{n}}\int_{T_n^*}
      \theta(x) \; \ell(dx)
\]
and we want to prove that $\Delta_n$ tends $a.s.$ to 0 (at least along
some subsequence).

Using the martingale of Equation \reff{eq:X2-mart}, we have that: 
\begin{equation}
   \label{eq:Xn-L2}
\N_\infty ^{(r)}\left[\Delta_n^2\; \Big | \; T_n \right] 
= \frac{2\alpha}{\sqrt{n}}\N_\infty ^{(r)}\left[\inv{\sqrt{n}}\int_{T_n^*}
      \theta(x) \; \ell(dx)\; \Big | \; T_n \right] .
\end{equation}

Then we use the following lemma whose proof is postponed to Section
\ref{sec:proof}. 
\begin{lem}\label{lem:maintheo}
Let $r>0$.
There exists a non-negative sequence $(R'_n,n\ge 1)$ of random
variables adapted to the the filtration $(\sigma(T_n),n\ge 1)$ and
which converges $\N_\infty^{(r)}$-a.s. to 0 such that, for all $n\ge
1$, $\N^{(r)}$-a.s.:
\begin{equation}
   \label{eq:CV-Ntheta}
r\N_\infty ^{(r)}\left[\inv{\sqrt{n}}\int_{T_n^*}
      \theta(x) \; \ell(dx)\; \Big | \; T_n \right] \leq 
\frac{L_n}{\sqrt{n}} \N_\infty ^{(r)}\left[\Theta \; \Big | \; T_n
  \right] + R'_n.
\end{equation}
\end{lem}

With this lemma, we have:
\begin{align*}
\N_\infty ^{(r)}\left[\sum_{n\geq 1} \Delta_{n^4} ^2\ind_{\{R'_{n^4}\leq
    1\}} \right] 
&= \sum_{n\geq 1} \N_\infty ^{(r)}\left[\N_\infty
  ^{(r)}\left[\Delta_{n^4} ^2 | \; T_{n^4} \right]\ind_{\{R'_{n^4}\leq
    1\}}\right]  \\ 
&\leq \sum_{n\geq 1}
  \frac{2\alpha}{n^2r} \N_\infty ^{(r)}\left[\left(\frac{L_{n^4}}{n^2}
      \N_\infty  
  ^{(r)}\left[\Theta \;  \Big| \; T_{n^4}
  \right]+R'_{n^4}\right)\ind_{\{R'_{n^4}\leq 
    1\}} \right]  \\
&\leq \sum_{n\geq 1} \frac{2\alpha}{n^2 r} \left(\frac{1}{n^2} \N_\infty
  ^{(r)}\left[L^2_{n^4} \right]^{1/2}
\N_\infty
  ^{(r)}\left[\Theta^2  \right]^{1/2}+1\right)
\\
&<+\infty ,
\end{align*}
where  we used  \reff{eq:Xn-L2}  and \reff{eq:CV-Ntheta}  for the  first
inequality,  Cauchy-Schwartz inequality  for the  second one,  and Lemma \ref{lem:densite-h} as well  as \reff{eq:moment-int-theta}  for  the
last one.
This result implies that  $\N_\infty ^{(r)}$-a.s. $\lim_{n\rightarrow +\infty }
\Delta_{n^4}\ind_{\{R'_{n^4}\leq    1\}}=0$    and    thus    $\N_\infty
^{(r)}$-a.s.  $\lim_{n\rightarrow  +\infty   }  \Delta_{n^4}=0$  as  the
sequence $(R'_n, n\geq 1)$  converges $\N^{(r)}$-a.s. to $0$.

\medskip
In order to conclude, it remains to study the asymptotic behavior of
$$\inv{\sqrt{n}}\int_{T^*_n} \theta(x)\; \ell(dx)$$
which is the purpose of the next proposition which will also be proven
in Section \ref{sec:proof}.

\begin{prop}
   \label{prop:maintheo}
We have that, for all $r>0$,  $\N^{(r)}_\infty $-a.s.:
\begin{equation}
   \label{eq:CV-theta}
\lim_{n\rightarrow +\infty } \inv{\sqrt{n}}\int_{T^*_n} \theta(x)\; \ell(dx)
= \inv{\sqrt{r 
    \alpha} } \Theta.
\end{equation}
\end{prop}

We deduce
from  \reff{eq:CV-theta},  that  $\N_\infty  ^{(r)}$-a.s.  the  sequence
$(X_{n^4}^*/n^2,         n\geq          1)$         converges         to
$2\sqrt{\frac{r}{\alpha}}\Theta$. Then using  that $(X_n^*, n\geq 1)$ is
increasing, we get for $k\in \N$, such that $n^4<k\leq (n+1)^4$, that:
\[
 \frac{n^2}{(n+1)^2}\; \frac{X^*_{n^4}}{n^2}\leq \frac{X^*_k}{\sqrt{k}}
\leq \frac{(n+1)^2}{n^2}\; \frac{X^*_{(n+1)^4}}{(n+1)^2}\cdot
\]
Thus,    we   get    that   $\N_\infty    ^{(r)}$-a.s.    the   sequence
$(X_{k}^*/\sqrt{k},         k\geq        1)$         converges        to
$2\sqrt{\frac{\alpha}{r}}\Theta$.

\subsection{Proof of Proposition \ref{prop:maintheo} and Lemma \ref{lem:maintheo}}
\label{sec:proof}

First,    let    us    remark    that,   as    $L_n/\sqrt    n$    tends
$\N_\infty^{(r)}$-a.s.     to     $\sqrt{r/\alpha}$     by     Corollary
\ref{cor:cv-Ln} and is $\sigma(T_n)$-measurable, it suffices to study the limit of
\[
\frac{1}{L_n}\int_{T_n^*}\theta(x)\, \ell(dx).
\]

Let us exhibit a martingale that converges to $\Theta$.
Let $\cf_n$  be the $\sigma$-field  generated by $T_n$  and $(\theta(x),
x\in  T_n)$.   The  filtration  $(\cf_n,  n\geq  1)$ is  increasing
towards  $\vee_{n\geq  1} \cf_n=\cf$,  the  $\sigma$-field generated  by
$\ct$ and $(\theta(s), s\in [0,\sigma])=(\theta(x), x\in \ct)$.

We consider the process $(M_n, n\geq 1)$ defined by, for $q\in [0, +\infty ]$:
\[
M_n=\N_q ^{(r)}\left[\Theta \;  \Big | \;
  \cf_n \right].
\]
Thanks to \reff{eq:moment-int-theta}, 
we get that:
\[
\N^{(r)}_q [M_n^2]
\leq  \N_q^{(r)}\left[\Theta^2
\right]
\leq   \N_\infty ^{(r)}\left[\Theta^2
\right]= \frac{r}{\alpha}\cdot 
\]
Therefore  $(M_n,  n\geq  1)$  is  (a well  defined)  square  integrable
non-negative     martingale.      In     particular     it     converges
$\N_q^{(r)}$-a.s.  (and   in  $L^2(\N_q^{(r)})$)  to   $\Theta$  as  the
increasing $\sigma$-fields $\cf_n$ increase to $\cf$.

In the next  lemma whose proof is given  in Section \ref{sec:proof-lem},
we compare $\frac{1}{L_n}\int_{T_n^*}\theta(x)\, \ell(dx)$ to $M_n$.
\begin{lem}
   \label{lem:majo-Mn}
We have, for $n\geq 1$, 
\begin{equation}
   \label{eq:Majo-Mn-VR}
-R_n \leq M_n - \frac{r}{L_n} \int_{T^*_n} \theta(x)\;\ell(dx) \leq  V_n,
\end{equation}
where     $(R_n,    n\geq     1)$ and  $(V_n, n\geq
1)$   are  non-negative   sequences  which   converge  $\N^{(r)}_\infty
$-a.s. to $0$. Furthermore the non-negative sequence $(R'_n, n\geq 1)$,
with $R'_n=\N_\infty ^{(r)}[R_n | T_n]\; L_n/\sqrt{n}$, converges 
 $\N^{(r)}_\infty
$-a.s. to $0$. 
\end{lem}

This lemma ends the proof of Proposition
\ref{prop:maintheo}. Moreover, as $\N_\infty ^{(r)}[M_n\; |\; T_n]=\N_\infty ^{(r)}[\Theta\; |\;
T_n]$, it also proves Lemma \ref{lem:maintheo}. 

\subsection{Proof of Lemma \ref{lem:majo-Mn}}
\label{sec:proof-lem}
In order to first give a
description of the marked tree
conditionally on $\cf_n$, we consider the sub-trees that are grafted on
$T_n$. For $x,y\in\ct$, we define an equivalence relation by
setting
$$x\sim_{T_n} y\iff \lb\emptyset,x\rb\cap T_n=\lb\emptyset,y\rb\cap T_n$$
and we set $(\ct_i,i\in I_n)$  for the different equivalent classes. The
set $\ct_i$  can be viewed  as a rooted  real tree with  root $x_i=\ct_i
\cap T_n$. Notice that $x_i$ represents the  point of $T_n$ at which the tree
$\ct_i$ is grafted on $T_n$.  Finally, we set $\theta_i=\theta(x_i)$ and
$\sigma_i=m^\ct(\ct_i)$ which corresponds to the length of the height
process of $\ct_i$. 

Using Theorem 3 of \cite{lg:urtbe} (combined with the spatial motion
$\theta$), we get the following result. 

\begin{lem}\label{lem:Poisson-decomp}
Under $\N_q$ conditionally on $\cf_n$, the point measure
$$\sum_{i\in I_n}\delta_{(\ct_i,\theta_i,x_i)}(d\ct,dq',dx)$$
is a Poisson point measure with intensity
$$2\alpha \ind_{T_n}(x)\ell(dx)\; \N[d\ct]\, \delta_{\theta(x)}(dq').$$
\end{lem}

We deduce from that Lemma the next result.
\begin{lem}
   \label{lem:buisson}
Under $\N_q^{(r)}$ and conditionally on $\cf_n$,  the point measure 
\[
\cn_n(d\sigma, dq', dx)=\sum_{i\in I_n} \delta_{(\sigma_i,\theta_i, x_i)}(d\sigma,dq',dx)
\]
is distributed as a Poisson point measure:
\[
\tilde \cn(d\sigma,dq',dx) =\sum_{j\in J} \delta_{(\tilde\sigma_j,\theta_j,
  x_j)}(d\sigma,dq',dx) 
\]
with intensity $2\alpha \ind_{T_n}(x) \ell(dx)\;
\frac{d\sigma}{2\sqrt{\alpha\pi}\;\sigma^{3/2}}\ind_{\{\sigma>0\}}\; 
\delta_{\theta(x)}(dq')$ 
conditioned on $\{\sum_{j\in J} \tilde \sigma_j=r\}$. 
\end{lem}

We can compute some elementary functionals of $\cn_n$. 
\begin{lem}
   \label{lem:E(r),Ln}
Under $\N_q^{(r)}$ and conditionally on $\cf_n$,  the  point
measure $\cn_n$ has intensity:
\[
2 \alpha \ind_{T_n}(x) \ell(dx)\;  \E^{(r),L_n}[d\sigma]\;  \delta_{\theta(x)}(dq),
\]
where $\E^{(r),L_n}$ satisfies, for any non-negative measurable function
$F$:
\[
2\alpha \int_{T_n} \ell(dx)\;  \E^{(r),L_n} [F(x,\sigma)]
=\E\left[\sum_{j\in J} F(s_j, \tilde \sigma_j)\Bigm| \sum_{j\in J}\tilde
\sigma_j=r\right].
\]
We also have:
\begin{equation}
   \label{eq:E-s*}
\E^{(r),L_n} [\sigma]=\frac{r}{2\alpha L_n}
\quad\text{and} \quad 
   \E^{(r),L_n} [\sigma^{3/2}]\leq \frac{2}{\sqrt{\alpha\pi}} \inv{L_n}
   r^2\expp{-\alpha L_n^2/r}.
\end{equation}
\end{lem}

\begin{proof}
 The first part of the Lemma is a consequence of the exchangeability of
 $(\sigma_i, i\in I_n)$. With $F(q,r')=r'$, we get:
\[
2\alpha L_n \E^{(r),L_n} [\sigma]
=2 \alpha \int_{T_n} \ell(dx)\;  \E^{(r),L_n} [\sigma]
=\E\left[\sum_{j\in J} \tilde \sigma_j\mid \sum_{j\in J}\tilde
\sigma_j=r\right]=r.
\]
This gives the first equality of \reff{eq:E-s*}. 
Recall that:
\[
\N\left[1-\expp{-\mu \sigma}\right]=\int_0^\infty
\frac{dr}{2\sqrt{\alpha\pi}\;r^{3/2}}\;  \left(1- \expp{-\mu
    r}\right)=\sqrt{\mu/\alpha}. 
\]
We have, using the Palm formula for Poisson point measures, for $a>1/2$:
\begin{align*}
\E\left[\sum_{j\in J} \tilde \sigma_j^a \expp{-\mu \sum_{i\in J}
    \tilde \sigma_i} \right]
& = \E\left[\sum_{j\in J} \tilde \sigma_j^a \expp{-\mu 
    \tilde \sigma_j} 
\expp{-\mu \sum_{i\in J, i\neq j}
    \tilde \sigma_i} \right]\\
& =2\alpha L_n\N\left[\sigma^a \expp{-\mu \sigma}\right]
\exp\left(-2\alpha L_n \N\left[1-\expp{-\mu\sigma}\right] \right)\\
& = 2\alpha L_n\N\left[\sigma^a \expp{-\mu \sigma}\right]
\expp{-2L_n\sqrt{\alpha\mu}}.
\end{align*}

Moreover, we have:
\[
\N\left[\sigma^a \expp{-\mu \sigma}\right]
=\int_0^\infty
\frac{dr}{2\sqrt{\alpha\pi}\;r^{3/2}}\;  r^a\expp{-\mu r}
= \inv{2\sqrt{\alpha\pi} }\Gamma(a-1/2) \mu ^{1/2-a}.
\]

We deduce that:
\begin{multline*}
   \E  \left[\sum_{j\in J} \left(2\sqrt{\alpha} L_n \tilde \sigma_j^{3/2}
       + 
       \inv{\Gamma(3/2)} \tilde \sigma_j^2\right)\expp{-\mu \sum_{i\in J}
    \tilde \sigma_i} \right]\\
   \begin{aligned}
&=2\alpha L_n\expp{-2L_n\sqrt{\alpha\mu}}\left(2\sqrt{\alpha}L_n
  \N[\sigma^{3/2}\expp{-\mu\sigma}]
  +\inv{\Gamma(3/2)}\N[\sigma^2\expp{-\mu\sigma}]\right)\\   
&=2\alpha L_n \expp{-2L_n \sqrt{\alpha\mu}}
  \frac{1}{2\sqrt{\alpha\pi}}\left(\frac{2\sqrt{\alpha}L_n}{\mu}
    +\inv{\mu^{3/2}}\right) \\ 
&=\frac{2}{\sqrt{\pi} }\frac{\partial^2}{\partial \mu^2}
\expp{-2L_n\sqrt{\mu\alpha}}. 
\end{aligned}
\end{multline*}
Let us recall the Laplace transform for the density of a stable
subordinator of index $1/2$: for $a>0$ and $\mu\ge 0$,
$$a\int_0^{+\infty}\frac{dr}{\sqrt{2\pi r^3}}\expp{-\mu
  r-a^2/(2r)}=\expp{-a\sqrt{2\mu}}.$$
>From that formula, we have 
\begin{align*}
\frac{\partial^2}{\partial \mu^2}
\expp{-2L_n\sqrt{\mu\alpha}} 
&=\frac{\partial^2}{\partial \mu^2}
\inv{\sqrt{\pi}} \int_0^{+\infty } \frac{dx}{x^{3/2}}
\expp{-1/x}\expp{-\alpha L_n^2\mu   x}\\
&=\inv{\sqrt{\pi}}\left({\alpha}{L_n^2}\right)^2
 \int_0^{+\infty } dx\; \sqrt{x}
\expp{-1/x}\expp{-\alpha L_n^2\mu   x}\\
&=\frac{L_n\sqrt\alpha}{\sqrt{\pi}}
 \int_0^{+\infty } dr\; \sqrt{r}
\expp{-\alpha L_n^2/r }\expp{-\mu r}\\
&=2\alpha L_n
 \int_0^{+\infty } \frac{dr}{2\sqrt{\alpha\pi}\; r^{3/2}}\;r^2  
\expp{-\alpha L_n^2/r }\expp{-\mu r}.
\end{align*}
We deduce that:
\[
   \E\left[\sum_{j\in J} \left(2\sqrt{\alpha} L_n \tilde \sigma_j^{3/2}
       + 
       \inv{\Gamma(3/2)} \tilde \sigma_j^2\right)\;\Big|\; \sum_{i\in J}
    \tilde \sigma_i=r \right]
=\frac{4\alpha L_n}{\sqrt{\pi}}r^2  
\expp{-\alpha L_n^2/r }.
\]
Then, using the first part of Lemma \ref{lem:E(r),Ln} with
$F(s,\sigma)=2\sqrt\alpha L_n\sigma^{3/2}+
\inv{\Gamma(3/2)}\sigma^2$, we get  the second inequality of \reff{eq:E-s*}.
\end{proof}

Now we prove  Lemma \ref{lem:majo-Mn}.

We consider the set  $I_n^*=\{i\in I_n,\ x_i\ge m_n\}$ of indexes such
that $\ct_i$ is not grafted on the edge $\lb\emptyset,m_n\rb$. We set:
\[
A_n=\{s\ge 0; \lb \emptyset,s\rb\cap T_n^*\ne \emptyset\}=\overline{\bigcup 
_{i\in I^*_n} T^i},\quad
M^*_n=\N_q^{(r)}\left[\int_{A_n} \!\!\theta(s)\; ds \; \Big|\; \cf_n
\right]
\;\; \text{and}\; \; 
V_n=M_n-M^*_n.
\]
Notice that the sequence $(A_n, n\in \N^*)$ is non-decreasing and that
$\bigcap _{n\in \N^*} A_n^c=\emptyset$, as there is no tree grafted on
the root. By dominated convergence, this implies that
$\N^{(r)}_q$-a.s.:
\[
\lim_{n\rightarrow+\infty } \int_{A_n^c} \theta(s)\; ds =0.
\]
As:
\[
V_{n+m} =\N_q^{(r)} \left[\int_{A_{n+m}^c} \theta(s)\; ds \; \Big|\; 
  \cf_{n+m}  \right]\leq  \N_q^{(r)}\left[\int_{A_n^c} \theta(s)\; ds
  \; \Big|\;   \cf_{n+m} \right], 
\]
and   as    $\cf_{n+m}   $   increases    to   $\cf$,   we    get   that
$\limsup_{m\rightarrow+\infty }  V_{n+m}\leq \int_{A_n^c} \theta(s)\;
ds $ and thus $\N^{(r)}_q$-a.s.
\begin{equation}
   \label{eq:lim-Vn}
\lim_{n\rightarrow+\infty } V_{n}=0.
\end{equation}

We define  the function $H_q$  (see Proposition \ref{prop:Nr-int-theta}
for a closed formula) by: 
$$H_q(r)=\N_q^{(r)}[\Theta].$$

 We  have, with
$\Theta_i=\Theta(\ct_i)=\int_{\ct_i} \theta(x)\;m^\ct(dx) $:
\begin{align*}
   M^*_n
= \N_q^{(r)}\left[\int_{A_n} \theta(s)\; ds \; \Big|\; \cf_n \right]
= \N_q^{(r)}\left[\sum_{i\in I_n^*} \Theta_i \; \Big|\; \cf_n \right]
&= \N_q^{(r)}\left[\sum_{i\in I_n^*} \N^{(\sigma_i)}_{\theta(x_i)}
  \left[\Theta\right]
\; \Big|\; \cf_n \right]\\
&= \N_q^{(r)}\left[\sum_{i\in I_n^*} H_{\theta(x_i)}(\sigma_i)
\; \Big|\; \cf_n \right].
\end{align*}
Since   $H_q(r)\leq    qr$,   see   \reff{eq:encadrement-int-theta}   in
Proposition \ref{prop:Nr-int-theta}, we get  using the first equality of
\reff{eq:E-s*} in Lemma \ref{lem:E(r),Ln}:
\[
M^*_n= 2\alpha \int_{T_n^*} \ell(dx) \; \E^{(r),L_n}[H_{\theta(x)} (\sigma)]
\leq  2\alpha \int_{T_n^*} \ell(dx) \; \theta(x) \E^{(r),L_n}[\sigma]=r \inv{L_n}
\int_{T_n^*} \ell(dx) \; \theta(x). 
\]
This gives the upper bound of \reff{eq:Majo-Mn-VR}.

We shall now prove the lower bound of \reff{eq:Majo-Mn-VR}.
Since $H_q(r)\geq qr - \inv{2} \sqrt{\alpha \pi}\; q^2 r^{3/2}$, see
\reff{eq:encadrement-int-theta} in Proposition \ref{prop:Nr-int-theta}, we
also get using  the second equality of \reff{eq:E-s*} in Lemma \ref{lem:E(r),Ln}:
\begin{align*}
M_n\geq M_n^* 
&\geq   r \inv{L_n}
\int_{T_n^*} \ell(dx) \; \theta(x) -\inv{2}\sqrt{\alpha\pi}\; 
\E^{(r),L_n}[\sigma^{3/2}]  \int_{T^*_n} 
\ell(dx) \; \theta(x) ^2 \\
& \geq r  \inv{L_n}
\int_{T_n^*} \ell(dx) \; \theta(x) - \inv{2}  r^2 \expp{-\alpha L^2_n/ r}
\theta_{\emptyset,n}^2 
\end{align*}
where $\theta_{\emptyset,n}=\theta(m_n)$.
This proves the lower bound of \reff{eq:Majo-Mn-VR} with:
\begin{equation}
   \label{eq:def-Rn}
R_n=\inv{2} r^2 \expp{-\alpha L^2_n/r} 
\theta_{\emptyset,n}^2.
\end{equation}
It remains to prove that this quantity tends to 0.
First, we have:
\[
\N_\infty^{(r)}[h_{\emptyset,n}^2\theta_{\emptyset,n}^2] 
=\N_\infty^{(r)}[h_{\emptyset,n}^2\N_\infty^{(r)}
[\theta_{\emptyset,n}^2\,|\,h_{\emptyset,n}]]  =\inv{(2\alpha)^2},
\]
where we used that $\theta_{\emptyset,n}$ is
exponentially distributed conditionally given $h_{\emptyset,n}$ for
the second equality. We deduce that:
\[
\N_\infty^{(r)}\left[\sum_{n=1}^{+\infty}\frac{h_{\emptyset,n}^2
    \theta_{\emptyset,n}^2}{n^2}\right]<\infty  
\]
and hence $\N_\infty^{(r)}$-a.s.:
\[
\sum_{n=1}^{+\infty}\frac{h_{\emptyset,n}^2\theta_{\emptyset,n}^2}{n^2}<\infty.
\]
This implies that, $\N_\infty^{(r)}$-a.s., for some finite
$\cf_n$-measurable random variable  $C_1$:
\[
h_{\emptyset,n}^2\theta_{\emptyset,n}^2\le C_1 n^2.
\]
Using  Lemma \ref{lem:moment-h0}, we have
$\N_\infty^{(r)}[h_{\emptyset,n}^{-1/2}]\sim n^{1/4} \sqrt{\alpha\pi/
  2r}$, 
which implies by similar  arguments that, $\N_\infty^{(r)}$-a.s., for
some finite $\sigma(T_n)$-measurable  random variable  $C_2$:  
\begin{equation}
   \label{eq:majo-h0n}
h_{\emptyset,n}^{-1/2}\le C_2 n^{3/2}.
\end{equation}
Finally, using Formula \reff{eq:def-Rn} for $R_n$, we have
$\N_\infty^{(r)}$-a.s.:  
\[
R_n\le C_1C_2^4\; n^8r^2\, \expp{-\alpha L_n^2/ r}.
\]
As $\N_\infty^{(r)}$-a.s. $\lim_{n\rightarrow +\infty } L_n/\sqrt{n}
=\sqrt{r/\alpha}$, we deduce that $\lim_{n\rightarrow +\infty } R_n=0$. 

Using \reff{eq:majo-h0n}, we deduce that:
\[
R'_n=\frac{L_n}{\sqrt{n}}\N_\infty ^{(r)} [R_n\; |\; T_n]
=\frac{L_n}{\sqrt{n}} \; \frac{r^2}{2}\expp{-\alpha L_n^2/r} \inv{4
  \alpha^2} \inv{h_{\emptyset,n}^2}
\leq C_2^4 \frac{r^2}{8\alpha^2} \; n^{11/2} L_n \expp{-\alpha L_n^2/r} .
\]
Thus,  we get that the non-negative  sequence $(R'_n,  n\geq  1)$, converges
$\N^{(r)}_\infty $-a.s. to $0$, 
which ends the proof.


\section{Appendix}
\label{sec:appendix}

\subsection{Computations on Rayleigh distributions}

Let    $Z$    be   a    Rayleigh    random    variable.
\begin{lem}
   \label{lem:Lap-Rayleigh}
  Let  $\mu>0$, $c\geq  0$. We have: 
\begin{equation}
   \label{eq:Lap-Rayleigh}
\inv{\sqrt{\pi}} \int_0^{+\infty } \frac{dr}{\sqrt{r}}\; \expp{-\mu r} 
\E\left[\expp{-\sqrt{2r}\; c Z'}\right]
= \inv{c+\sqrt{\mu}}\cdot
\end{equation}
\end{lem}

\begin{proof}

We set
\[
J=\sqrt{\frac{\mu}{2}}\int_0^\infty
\frac{dr}{\sqrt{r}}\;  \expp{-\mu 
  r} \int_0^\infty  dx\; x\expp{-x^2/2} \expp{-c\sqrt{2 r} \; x}.
\]
With the change of variable $t^2=2\mu r$ and with
$\rho=c/\sqrt{\mu}$, we get: 
\begin{align*}
J
&=   \int_{[0, +\infty )^2} dt dx\; x\; \exp(-(t^2+x^2+2\rho tx)/2)\\
&=   \int_{[0, +\infty )^2} dt dx\; (x+ \rho  t)\;
\expp{-(t^2+x^2+2\rho tx)/2}  - \rho\int_{[0, +\infty )^2} dt
dx\;t\; \expp{-(t^2+x^2+2\rho tx)/2}  \\
&= \int_0^\infty dt \left[-\exp(-(t^2+x^2+2\rho tx)/2)
\right]_{x=0}^{x=+\infty } - \rho J\\
&= \int_0^\infty dt\; \expp{-t^2/2} - \rho J\\ 
&=\sqrt{\pi/2} -\rho J.
\end{align*}
This implies that $\displaystyle 
J=\frac{\sqrt{\pi/2}}{\rho+1}
=\sqrt{\frac{\mu}{2}}
\frac{\sqrt{\pi}}{c+\sqrt{\mu}}$,
which is exactly what we needed. 
\end{proof}

\subsection{Joint law of $(\Theta,\sigma)$}
Notice that the joint law of $(\Theta,\sigma)$ under $\N_\infty $ is given in
Corollary \ref{prop:Laplace-theta}. However, we shall need the joint
distribution under $\N_q$. For this reason, we compute the Laplace
transform of $(\Theta,\sigma)$ using the theory of 
super-process.

Let $\lambda> 0$, $\mu\geq 0$.
We set for $q\in[0,+\infty]$:
\begin{equation}
   \label{eq:def-F}
F(q)=\N_q\left[1-\expp{-\lambda\Theta-\mu\sigma}\right].
\end{equation}
We define the function:
\begin{equation}
   \label{eq:G2}
G(x)=\left(\sqrt{\frac{\mu}{\alpha}}+\frac{\lambda}{2\alpha}
\right)\expp{\frac{2\alpha}{\lambda}\left(x-\sqrt{\mu/\alpha}\right)} -x -\frac{\lambda}{2\alpha}\cdot 
\end{equation}
The function $G$ is one-to-one from $[\sqrt{\mu/\alpha},+\infty )$ to
$[0,+\infty )$, is increasing and is of class $\cc^\infty $. 

\begin{lem}
   \label{lem:eqF}
   Let $\lambda> 0$,  $\mu\geq 0$. The function $F$  is of class $\cc^1$
   on $[0,+\infty )$ 
   and solves the following equation on $[0, +\infty )$:
\begin{equation}
   \label{eq:int-Ff}
\alpha F(q)^2+ 2\alpha\int_0^q xF'(x)\; dx = \lambda q+\mu.
\end{equation}
Furthermore, we have $F=G^{-1}$. 
\end{lem}

\begin{proof}
  The  first part  of the  Lemma  is a  well known  result from  Laplace
  transform of superprocess  \cite{d:dpde} (Theorem 1.8) or equivalently
  of   Brownian   snake   \cite{lg:sbprspde}   (Theorem  4).    We   set
  $f(x)=\lambda x+\mu$.  We introduce the function  $u_t(q)$ defined for
  $t\geq 0$ and $q\geq 0$ by:
\[
u_t(q)=\N_q\left[1-\expp{-\int_0^\sigma f(\hat\theta_s)\ind_{\{\zeta_s\leq
      t\}}\; ds}\right]. 
\]
We deduce  from Theorem  II.5.11 of \cite{p:dwsmvd} that $u$  is the
unique non-negative solution of:
\[
u_t(q)+\E_q\left[\int_0^t \alpha u_{t-s}(\theta(s))^2\; ds\right]=
\E_q\left[\int_0^t f(\theta(s))\; ds\right].
\]
Using the Markov property of $\theta$, we get that for $t\geq r\geq 0$:
\begin{equation}
   \label{eq:ut}
u_t(q)+\E_q\left[\int_0^r \alpha u_{t-s}(\theta(s))^2\; ds\right]=
\E_q\left[\int_0^r f(\theta(s))\; ds\right] +\E_q[u_{t-r}(\theta(r))].
\end{equation}
Notice that $\lim_{t\rightarrow +\infty }
u_t(q)=F(q)$. And we have:
\[
u_t(q)\leq F(q)\leq \N\left[1-\expp{-(q+\mu)\sigma}
\right]=\sqrt{(q+\mu)/\alpha}.
\]
By monotone convergence, we deduce from \reff{eq:ut} that:
\[
F(q) +\E_q\left[\int_0^r \alpha F(\theta(s))^2\; ds\right]=
\E_q\left[\int_0^r f(\theta(s))\; ds\right] +\E_q[F(\theta(r))].
\]
This implies that the process $N=(N_t, t\geq 0)$ defined by:
\[
N_t=F(\theta(t)) +\int_0^t \left(f(\theta(s)) - \alpha F(\theta(s))^2
\right)\; ds,
\]
is  a  martingale  under  $\E_q$,  for  $q<+\infty  $.  We  deduce  from
\reff{eq:def-Mg} (with $g=F$) that:
\[
\int_0^t \left(f(\theta(s)) - \alpha F(\theta(s))^2
- 2\alpha\int_0^{\theta(s)} (F(x)-F(q))\; dx \right)\; ds
\]
is a martingale. Since it is predictable, it is a.s. constant. 
We get that a.e. for $q\geq 0$:
\[
f(q) -\alpha F(q)^2 +2\alpha q F(q) -2\alpha\int_0^q F(x)\; dx =0,
\]
that is a.e.:
\[
F(q)=\sqrt{q^2 -2 \int_0^q F(x)\; dx +(f(q)/\alpha)}\; +q. 
\]
Since by construction $F$ is non-decreasing, we get that $F$ is
continuous and then of class $\cc^1$. An obvious integration by parts gives
 \reff{eq:int-Ff}. 

We now prove the second part of the Lemma. 
Notice that $F(0)=\N_0\left[1-\expp{-\mu
    \sigma}\right]=\sqrt{\mu/\alpha}$. By differentiating \reff{eq:int-Ff}
we have:
\begin{equation}
   \label{eq:dF}
2\alpha F'(q)(F(q)+ q)=\lambda.
\end{equation}
This implies that $F'>0$ and thus $F$ is one-to-one from $[0,+\infty )$
to $[\sqrt{\mu/\alpha}, +\infty )$. Moreover,  $F^{-1}$ solves the
  differential equation
\begin{equation}
   \label{eq:EDO-G}
g'(x)=\frac{2\alpha}{ \lambda}(g(x)+x).
\end{equation}
Elementary computations give that the unique solution to \reff{eq:EDO-G}
with the initial condition $g(\sqrt{\mu/\alpha})=0$ is $G$. Thus, we
get by uniqueness $F^{-1}=G$. 
\end{proof}

Notice that $F(+\infty )=+\infty $ which doesn't able us to compute
directly the
Laplace transform of $(\Theta,\sigma)$. However, we deduce easily the
following result, which gives an alternative proof of Corollary
\ref{prop:Laplace-theta}.

\begin{cor}
   \label{cor:partilal-F-infity}
Let $\lambda> 0$, $\mu\geq 0$. We have:
\begin{equation}
   \label{eq:LaplaceN}
\N_\infty\left[\sigma \expp{-\mu \sigma -\lambda \Theta} \right]
=\inv{2\sqrt{\alpha\mu}+\lambda}\cdot
\end{equation}
In particular, under $\N_\infty $, conditionally on $\sigma$,
$\sqrt{\frac{2\alpha} {\sigma}}\;  \Theta$
is distributed as a Rayleigh random variable $Z$ independent of
$\sigma$. 
\end{cor}

\begin{proof}
   We have for $q\in[0,+\infty)$:
\begin{equation}
   \label{eq:dF-mu}
\partial_\mu F(q)=\N_q\left[\sigma \expp{-\lambda\Theta-\mu\sigma}
\right] .
\end{equation}
Since $G(F(q))=q$ we get:
\[
(\partial_\mu G)(F(q)) + G'(F(q))\; \partial_\mu F(q)=0.
\]
We have:
\[
\partial_\mu G(x)=
-\frac{1}{\lambda}\expp{\frac{2\alpha}{\lambda}(x-\sqrt{\mu/\alpha})}=-\frac{1}{\lambda}\;
\frac{1}{2\sqrt{\alpha\mu}+\lambda}(2\alpha
G(x)+2\alpha x+\lambda).
\]
Notice that $G'(F(q))=1/F'(q)$. We deduce from \reff{eq:dF} that:
\begin{align*}
 \partial_\mu F(q) 
&=
 \frac{1}{2\alpha(F(q)+q)}\frac{1}{2\sqrt{\alpha\mu}+\lambda}(2\alpha
 q+2\alpha F(q)+\lambda)\\
& =\frac{1}{2\sqrt{\alpha\mu}+\lambda}\left(1+\frac{\lambda}{2\alpha(F(q)+q)}\right).
\end{align*}
Letting $q$ go to infinity gives the first part of the Corollary. 

For the last part, use Lemma \ref{lem:Lap-Rayleigh} and the distribution
of $\sigma$ under $\N$ given in \reff{eq:densite-s} to conclude. 
\end{proof}

The last part of the Section is devoted to the computation of the first
moment of $\Theta$ under $\N^{(r)}_q$, with $q<+\infty $. 
\medskip
We first give the asymptotic expansion of $F$ with respect to small $\lambda$. 
We write  $O(\lambda^k)$ for any function $g$ of $q,\mu$ and
$\lambda$ such that for any $q>0$, $\mu>0$ and $\varepsilon>0$ there
exists a finite constant $C$ (depending on $q$, $\mu$ and $\varepsilon$)
such that for all  $\lambda\in
[0,\varepsilon]$, $|g(q,\mu,\lambda)|\leq  C \lambda^k$. Notice that
$O(\lambda^k)$ is not uniform in $q$ or $\mu$. 

\begin{lem}
   \label{lem:dev-F-lambda}
Let $q\in (0,+\infty )$. 
We set $z=q\sqrt{\frac{\alpha}{\mu}}$. We have:
\begin{equation}
   \label{eq:dev-F-lambda}
F(q)= \sqrt{\frac{\mu}{\alpha}} + \frac{\lambda}{2\alpha} \log(1+z)- \frac{\lambda^2}{4\alpha^{3/2}\mu^{1/2}}
\;  \frac{z -\log(1+z)}{1+z} + O(\lambda^3).
\end{equation}
In particular, we deduce that:
\begin{equation}
   \label{eq:deriv-F}
\partial_\lambda F(q)_{|\lambda=0}=
\inv{2\alpha}\log(1+ z)
\quad\text{and}\quad
\partial^2_\lambda F(q)_{|\lambda=0}= - \frac{1}{2\alpha^{3/2}\mu^{1/2}}
\;
\frac{z -\log(1+z)}{1+z}\cdot
\end{equation}
\end{lem}

\begin{proof}
   Using the second part of Lemma \ref{lem:eqF} and \reff{eq:G2}, we get:
\begin{equation}
   \label{eq:del-F}
F(q)=\sqrt{\frac{\mu}{\alpha}} +
\frac{\lambda}{2\alpha}\log\left(\frac{2\alpha q+2\alpha F(q)+\lambda}{2\sqrt{\alpha\mu}+\lambda}\right).
\end{equation}
Using \reff{eq:def-F}, we get that $F(q)$ decreases to
$\sqrt{\mu/\alpha}$ when $\lambda $ goes down to $0$, that is
$F(q)=\sqrt{\mu/\alpha} +O(1)$. Plugging this in the right-hand side of
\reff{eq:del-F}, we get: 
\[
F(q)=\sqrt{\frac{\mu}{\alpha}} + O(\lambda).
\]
Plugging this in the right-hand side of
\reff{eq:del-F}, we get: 
\[
F(q) = \sqrt{\frac{\mu}{\alpha}} + \frac{\lambda}{2\alpha}\log(1+z) + O(\lambda^2).
\]
Plugging this again in the right-hand side of
\reff{eq:del-F}, we get \reff{eq:dev-F-lambda}. This readily implies
\reff{eq:deriv-F}.
\end{proof}

We can then compute the first moment of $\Theta$ under $\N^{(r)}_q$. 
\begin{prop}
   \label{prop:Nr-int-theta}
Let $H_q(r)=\N^{(r)}_q \left[\Theta \right]$. We have, for $r>0$ and $q\in [0, +\infty )$:
\begin{equation}
   \label{eq:def-H}
H_q(r)=\sqrt{\frac{r}{2\alpha}} \; \int_0^{q\sqrt{2\alpha r}} dy\;
\E\left[\expp{-yZ} \right]. 
\end{equation}
and
\begin{equation}
   \label{eq:encadrement-int-theta}
0\leq qr - H_q(r) \leq  \inv{2}\sqrt{\pi\alpha}\;  q^2 r^{3/2} .
\end{equation}
\end{prop}

\begin{proof}
By the change of variable $y=q\sqrt{2\alpha z}$, we have
$$H_q(r)=\frac{q\sqrt{r}}{2}\int_0^r\frac{dz}{\sqrt{z}}\int_0^{+\infty}dx\,x\,\exp\left(-\frac{x^2}{2}-q\sqrt{2\alpha
    z}\, x\right).$$
Then we compute for $\mu>0$,
\begin{align*}
\int_0^{+\infty}\frac{dr}{2\sqrt{\alpha\pi r}}\expp{-\mu r}H_q(r)
& =\frac{q}{4\sqrt{\pi\alpha}}\int_0^{+\infty}dr\expp{-\mu
      r}\int_0^r\frac{dz}{\sqrt{z}}\int_0^{+\infty}dx\,x\,\exp\left(-\frac{x^2}{2}-q\sqrt{2\alpha
      z}\, x\right)\\
& =\frac{q}{4\sqrt{\pi\alpha}}\int_0^{+\infty}\frac{dz}{\sqrt
  z}\int_z^{+\infty}dr\expp{-\mu
  r}\int_0^{+\infty}dx\,x\,\exp\left(-\frac{x^2}{2}-q\sqrt{2\alpha 
      z}\, x\right)\\
& =\frac{q}{4\sqrt{\pi\alpha}}\inv{\mu}\int_0^{+\infty}\frac{dz}{\sqrt
  z}\expp{-\mu z}\int_0^{+\infty}dx\,x\,\exp\left(-\frac{x^2}{2}-q\sqrt{2\alpha
      z}\,x\right)\\
& =\frac{q}{4\sqrt\alpha}\inv{\mu}\frac{1}{\sqrt\mu+q\sqrt \alpha},
\end{align*}
where we used  equality \reff{eq:Lap-Rayleigh} for the last equality. 

On the other hand, we have:
\begin{align*}
   \int_0^{+\infty }\frac{dr}{ 2\sqrt{\alpha\pi r}}\;\expp{-\mu r}
   \N^{(r)}_q \left[ \Theta \right]
&= -\partial_\mu \int_0^{+\infty }\frac{dr}{ 2\sqrt{\alpha\pi} \; r^{3/2}
    }\;\expp{-\mu r} \N^{(r)}_q \left[ 
\Theta \right]\\ 
&= -\partial_\mu 
\N_q \left[\expp{-\mu \sigma} 
     \Theta  \right]\\
&= -\partial_\mu \left[\partial_\lambda F(q)_{|\lambda=0}\right]\\
&= -\inv{2\alpha} \partial_\mu \log\left( 1+q\sqrt{\frac{\alpha}{\mu}}\right)\\
&=\frac{q}{4\sqrt{\alpha}} \;\inv{\mu}\; \inv{\sqrt{\mu}+
  q\sqrt{\alpha}} ,
\end{align*}
where we used Definition \reff{eq:def-F} of $F$ for the third equality and \reff{eq:deriv-F}
for the fourth one. 
Therefore, we have that  $dr$-a.e.   $\N^{(r)}_q  \left[
\Theta \right]= H_q(r)$. Then the equality
holds for all $r>0$ by continuity (using again a scaling argument). 

Then, use  $0\leq  1-\expp{-z}\leq  z$ for $z\geq 0$, to get
\reff{eq:encadrement-int-theta}. 
\end{proof}




\end{document}